\documentclass[11pt, reqno]{amsart}
\usepackage{amssymb,latexsym,amsmath,amsfonts,mathdots,enumitem}
\usepackage{tikz}
\usetikzlibrary{arrows.meta,positioning}
\usepackage{latexsym}
\usepackage[mathscr]{eucal}
\usepackage{soul,xcolor}
\usepackage{colortbl,xcolor}
\usepackage{lmodern}
\usepackage{sansmathaccent}
\usepackage[utf8]{inputenc}
\usepackage{tikz}
\usetikzlibrary{shapes,arrows}
\usetikzlibrary{matrix,calc,shapes,arrows,positioning}

\numberwithin{equation}{section}
\theoremstyle{plain}

\newtheorem{thm}{Theorem}[section]

\newtheorem{cor}[thm]{Corollary}

\newtheorem{prop}[thm]{Proposition}
\theoremstyle{definition}
\newtheorem{defn}[thm]{Definition}
\newtheorem{rem}[thm]{Remark}

\numberwithin{equation}{section}

\def\Re{\operatorname{Re}}

\def\beq{\begin{eqnarray}}
\def\eeq{\end{eqnarray}}
\def\beqa{\begin{eqnarray*}}
\def\eeqa{\end{eqnarray*}}

\def\beqn{\begin{equation}}
\def\eeqn{\end{equation}}

\def\mg#1{}

\renewcommand{\epsilon}{\varepsilon}
\renewcommand{\phi}{\varphi}

\begin{document}
\title[Few maps in the rich structure]{Few maps in the rich structure for the domains $G_{E(3;3;1,1,1)}$ and $G_{E(3;2;1,2)}$}
\author{Dinesh Kumar Keshari, Shubhankar Mandal and Avijit Pal}
\address[ D. K. Keshari]{School of Mathematical Sciences, National Institute of Science Education and Research Bhubaneswar, An OCC of Homi Bhabha National Institute, Jatni, Khurda,  Odisha-752050, India}
\email{dinesh@niser.ac.in}
\address[S. Mandal]{Department of Mathematics, IIT Bhilai, 6th Lane Road, Jevra, Chhattisgarh 491002}
\email{S. Mandal:shubhankarm@iitbhilai.ac.in}

\address[A. Pal]{Department of Mathematics, IIT Bhilai, 6th Lane Road, Jevra, Chhattisgarh 491002}
\email{A. Pal:avijit@iitbhilai.ac.in}



\begin{abstract}
  The primary goal of a rich structure for some naturally occurring domains $\mathcal X$ is to connect four naturally occurring objects of analysis in the context of $3\times 3$ analytic matrix functions on $\mathbb D$. Combining this rich structure with the classical realisation formula and Hilbert space models in the sense of Agler, one can effectively construct functions in the space $\mathcal O(\mathbb D,\mathcal X)$ of analytic maps from $\mathbb D$ to $\mathcal X$. This allows one to obtain solvability criteria for two cases of the $\mu$-synthesis problem. We describe few maps in the rich structure. We define $SE$ map  between $\mathcal S_{1}(\mathbb C^3,\mathbb C^3)$ and $\mathcal S_{3}(\mathbb C,\mathbb C)$ and establish the relation between $\mathcal{S}_{1}(\mathbb C^3,\mathbb C^3)$ and the set of analytic kernels on $\mathbb{D}^{3}$. We obtain the $UW$ procedure and using the $UW$ procedure we construct  the $Upper \,\,W$ and $Upper\,\ E$ maps. We also construct $Right~S$ and $SE$ maps. We show how the interpolation problems for $G_{E(3;3;1,1,1)}$ and $G_{E(3;2;1,2)}$ can be reduced to a standard matricial Nevanlinna-Pick problem.

\end{abstract}
\maketitle
\vskip-.5cm

\section{{\bf{Introduction}}}

\pagenumbering{arabic}
Let  $\Omega_1,\Omega_2\subset \mathbb C^n$ be domains in $\mathbb C^n.$  Suppose $\mathcal{O}\left(\Omega_{1}, \Omega_{2}\right)$ is used to represent the set of all holomorphic functions from $\Omega_{1} $ to $\Omega_{2}$. 
We discuss the classic Nevanlinna-Pick problem (see for instance \cite{AglerM, AglerMY}).

\textit{The Nevanlinna-Pick problem:} Assume that $z_1,\ldots, z_n \in \mathbb{D}$ and $\lambda_1,\ldots, \lambda_n \in \mathbb{D}$, whether there is an analytic function $F:\mathbb{D}\rightarrow\mathbb{C}$ such that $F(z_i)=\lambda_i$, $i=1,\dots,n$ and $|F(z)|\leq1$ for all $z\in \mathbb{D}$?

A necessary and sufficient condition to solve this problem was found by G.Pick in $1916.$ He showed that this problem is solvable if and only if the following pick matrix   $${\left(\left(\frac{1-\bar{\lambda_i}\lambda_j}{1-\bar{z_i}z_j}\right)\right)_{i,j=1}^n}$$ is positive semi-definite. Unaware of Pick's work, R. Nevanlinna also considered the problem in $1919.$ Later, while dealing with robust stabilization problem John Doyle introduced the $\mu$-synthesis problem (see for instance \cite{ds,jcd,Doyle,aj}). The $\mu$-synthesis problem is an interpolation problem for analytic matrix functions on the disc which are subject to a boundedness condition. It is a generalisation of the Nevanlinna-
Pick problem. 

Let $\mathcal M_{n\times n}(\mathbb{C})$ be the set of all $n\times n$ complex matrices and  $E$ be a linear subspace of $\mathcal M_{n\times n}(\mathbb{C}).$ We define the structured singular value (see \cite{Doyle}) $\mu_{E}: \mathcal M_{n\times n}(\mathbb{C}) \to [0,\infty)$ as follows:
\begin{equation}\label{mu}
	\mu_{E}(A):=\frac{1}{\inf\{\|X\|: \,\ \det(1-AX)=0,\,\, X\in E\}},\;\; A\in \mathcal M_{n\times n}(\mathbb{C})
\end{equation}
with the understanding that $\mu_{E}(A):=0$ if $1-AX$ is  nonsingular for all $X\in E$ \cite{ds, jcd}.   Here $\|\cdot\|$ denotes the operator norm. 

\textit{The $\mu_E$-synthesis problem (see \cite{omar})}: Given distinct points $\lambda_j \in \mathbb{D}$ and $W_j\in \mathcal{M}_{n\times n}(\mathbb{C})$, $j=1,\dots,k$. Construct, if possible an analytic $n\times n$ matrix-valued function $F$ on $\mathbb{D}$ such that $F(\lambda_j) = W_j$ for $j = 1,\dots,k$ and $\mu_{E}(F(\lambda))\leq 1$ for all $\lambda \in \mathbb{D}$.

Note that for $n = m = 1$, the $\mu_E$-synthesis problem is the classical Nevanlinna-Pick problem. Moreover, in the case $n = m$ and $E=\{c\mathbb{I}_n:c\in \mathbb{C}\}$, $\mu_{E}$ is the spectral radius $r$ and the $\mu_E$-synthesis problem is the spectral Nevanlinna-Pick problem (see \cite{omar}). 

\textit{The spectral Nevanlinna-Pick problem(\textasteriskcentered):}  Given distinct points $\lambda_j\in \mathbb{D}$ and $W_j\in \mathcal{M}_{k\times k}(\mathbb{C})$, $j=1,\dots,n$. Construct, if possible an analytic $k\times k$ matrix-valued function $F$ on $\mathbb{D}$
such that $F(\lambda_j) = W_j$ for $j = 1,\dots,n$ and $r(F(\lambda))\leq 1$ for all $\lambda \in \mathbb{D}$. Here $r$ is the spectral radius.

 In the attempt to solve the two-by-two spectral Nevanlinna-Pick problem, Agler and Young introduced in \cite{nj} a domain in $\mathbb{C}^2$ known as the \textit{symmetrised bidisc} $\Gamma_{E(2;1;2)}$. Berocovici and Fioas characterized the existence of an interpolant in $(\textasteriskcentered)$. Their characterization involves a nontrivial search for $m$ appropriate matrix in $\text{GL}_n(\mathbb C).$ Agler and Young\cite{JAY}, Costara\cite{Ccostara} proved that in the case $W_1,\ldots,W_m$ are non-derogatory then $(\textasteriskcentered)$ is equivalent to an interpolation problem from $\mathbb D$ to symmetrized polydisc. The problem $(\textasteriskcentered)$ descends to a region of much lower dimension with many pleasant properties. By taking $E=\{\operatorname{diag}[z_{1},z_{2}]\,\, \in \mathbb{C}^{2\times 2}: z_{1},z_{2}\in \mathbb{C}\},$ Abouhajar, White and Young \cite{Abouhajar} showed that $(\textasteriskcentered)$ is equivalent to an interpolation problem from $\mathbb D$ to tetrablock. By considering $E=\{\operatorname{diag}[z_{1}I_{n-1},z_{2}]\,\, \in \mathbb{C}^{n \times n}: z_{1},z_{2}\in \mathbb{C}\},$ Bharali in \cite{bha}, established that $(\textasteriskcentered)$ is equivalent to an interpolation problem from $\mathbb D$ to $\mathbb E_n.$ For this case also, the problem $(\textasteriskcentered)$ descends to a region of much lower dimension.
 
In 2017, Brown et al. obtain, for certain naturally occuring domains $\mathcal X$(symmetrized bidisc and tetrablock), a rich structure of interrelations among four naturally arising objects of analysis associated with $2\times 2$ analytic matrix functions on $\mathbb{D}.$ This rich structure, together with the classical realization formula and Hilbert space model in the sense of Agler, provided an effective framework for constructing functions in the space $\mathcal O(\mathbb{D},\bar{\mathcal X})$, and thereby yielding solvability criteria for $\mu$-synthesis problem. Rich structure for a domain $\mathcal X$ is encapsulated in the diagram see(\cite[page 1706]{BLY}). 

 Let  $E(n;s;r_{1},\dots,r_{s})\subset \mathcal M_{n\times n}(\mathbb{C})$ be the vector subspace comprising block diagonal matrices, defined as follows:
\begin{equation}\label{ls}
	E=E(n;s;r_{1},...,r_{s}):=\{\operatorname{diag}[z_{1}I_{r_{1}},....,z_{s}I_{r_{s}}]\in \mathcal M_{n\times n}(\mathbb{C}): z_{1},...,z_{s}\in \mathbb{C}\},
\end{equation}
where $\sum_{i=1}^{s}r_i=n.$  We recall the definitions of $\Gamma_{E(3; 3; 1, 1, 1)}$, $\Gamma_{E(3; 2; 1, 2)}$ and $\Gamma_{E(2; 2; 1, 1)}$, \cite{Abouhajar, bha, pal1}. The sets $\Gamma_{E{(2;2;1,1)}}$, $\Gamma_{E(3; 3; 1, 1, 1)}$ and $\Gamma_{E(3; 2; 1, 2)}$ are defined as 
 \begin{equation*}
		\begin{aligned}
			\Gamma_{E{(2;2;1,1)}}:=&\Big \{\textbf{x}=(x_1=a_{11}, x_2=a_{22}, x_3=a_{11}a_{22}-a_{12}a_{21}=\det A)\in \mathbb C^3: A\in \mathcal M_{2\times 2}(\mathbb C)\\~&{\rm{and}}~\mu_{E(2; 2; 1, 1)}\leq 1\Big \},
		\end{aligned}
\end{equation*}	
\begin{equation*}
	\begin{aligned}
		\Gamma_{E{(3;3;1,1,1)}}:=\Big \{\textbf{x}=(x_1=a_{11}, x_2=a_{22}, x_3=a_{11}a_{22}-a_{12}a_{21}, x_4=a_{33}, x_5=a_{11}a_{33}-a_{13}a_{31},&\\ x_6=a_{22}a_{33}-a_{23}a_{32},x_7=\det A)\in \mathbb C^7: A\in \mathcal M_{3\times 3}(\mathbb C)~{\rm{and}}~\mu_{E(3;3;1,1,1)}(A)\leq 1\Big \}
	\end{aligned}
\end{equation*}	
$${\rm{and}}$$  
\begin{equation*}
	\begin{aligned}
		\Gamma_{E(3;2;1,2)}:=\Big\{( x_1=a_{11},x_2=\det \left(\begin{smallmatrix} a_{11} & a_{12}\\
			a_{21} & a_{22}
		\end{smallmatrix}\right)+\det \left(\begin{smallmatrix}
			a_{11} & a_{13}\\
			a_{31} & a_{33}
		\end{smallmatrix}\right),x_3=\operatorname{det}A, y_1=a_{22}+a_{33}, &\\ y_2=\det  \left(\begin{smallmatrix}
			a_{22} & a_{23}\\
			a_{32} & a_{33}\end{smallmatrix})\right)\in \mathbb C^5
		:A\in \mathcal M_{3\times 3}(\mathbb C)~{\rm{and}}~\mu_{E(3;2;1,2)}(A)\leq 1\Big\}.
	\end{aligned}
\end{equation*}

The domains $\Gamma_{E(3; 2; 1, 2)}$ and $\Gamma_{E(2; 2; 1, 1)}$  are referred to as as $\mu_{1,3}-$\textit{quotient} and tetrablock, respectively \cite{Abouhajar, bha}. We now describe these objects of the rich structure related to the domains $\Gamma_{E(3; 3; 1, 1, 1)}$ and $\Gamma_{E(3; 2; 1, 2)}$ as follows: 
\begin{enumerate}
	\item $\mathcal S_{1}(\mathbb C^3,\mathbb C^3)$ denotes the $3\times3$ matricial Schur class of the disc, that is, the set of analytic $3\times3$ matrix function $F$ on $\mathbb{D}$ such that $\|F(\lambda)\|\leq 1$ for all $\mathbb{D}$;
	\item  $\mathcal S_{3}(\mathbb C,\mathbb C)$ is the Schur class of tri-disc $\mathbb{D}^3$, that is, $\mathcal O(\mathbb{D}^3,\mathbb{D})$;
	\item $\tilde{\mathcal{R}_{1}}$(in the case of $\Gamma_{E(3; 3; 1, 1, 1)}$) and $\tilde{\mathcal{S}_{1}}$(in the case of $\Gamma_{E(3; 2; 1, 2)}$). These are defined as follows: $$\begin{aligned}
		\tilde{\mathcal{R}_{1}}:=&\{(N^{(1)},N^{(2)},N^{(3)}):N^{(1)},N^{(2)},N^{(3)},K_{(N^{(1)},N^{(2)},N^{(3)})}~\text{are the analytic kernel on }\\ &\mathbb{D}^{3} ~\text{and} \,\,K_{(N^{(1)},N^{(2)},N^{(3)})}\,\, \text{is of rank 1}\},
	\end{aligned}$$
	 $$\begin{aligned}
		\tilde{\mathcal{S}_{1}}:=&\{(N^{(1)}+N^{(2)},N^{(3)}):N^{(1)},N^{(2)},N^{(3)},K_{(N^{(1)},N^{(2)},N^{(3)})}~\text{are the analytic kernel on }\\ &\mathbb{D}^{2} ~\text{and} \,\,K_{(N^{(1)},N^{(2)},N^{(3)})}\,\, \text{is of rank 1}\};~\text{and}
	\end{aligned}$$
	\item $\mathcal O(\mathbb D,\bar{\mathcal X})$, where $\mathcal X$ is either $\Gamma_{E(3; 3; 1, 1, 1)}$ or $\Gamma_{E(3; 2; 1, 2)}$.
\end{enumerate}

We aim to define a few maps in the rich structure. In Section 2, we define $SE$ map  between $\mathcal S_{1}(\mathbb C^3,\mathbb C^3)$ and $\mathcal S_{3}(\mathbb C,\mathbb C)$. Whereas, in Section 3, we establish the relation between $\mathcal{S}_{1}(\mathbb C^3,\mathbb C^3)$ and the set of analytic kernels on $\mathbb{D}^{3}$. We obtain "the $UW$ procedure" and using the $UW$ procedure we construct  the $Upper \,\,W$ and $Upper\,\ E$ maps in Section 4. We construct $Right~S$ and $SE$ maps in this section. Furthermore, Section 5 deals with how the interpolation problems $G_{E(3;3;1,1,1)}$ and $G_{E(3;2;1,2)}$ may be reduced to a standard matricial Nevanlinna-Pick problem.
  
\section{Maps between the sets $\mathcal S_{1}(\mathbb C^3,\mathbb C^3)$ and $\mathcal S_{3}(\mathbb C,\mathbb C)$}
In this section, our aim is to establish the relation between the sets $\mathcal S_{1}(\mathbb C^3,\mathbb C^3)$ and $\mathcal S_{3}(\mathbb C,\mathbb C).$ Suppose $\mathcal H_i=\mathbb C=\mathcal K_i$ for $i=1,2,3,$ $P=F=((F_{ij}))_{i,j=1}^{3}\in\mathcal S_{1}(\mathbb C^3,\mathbb C^3),$ and $X=\left(\begin{smallmatrix}z_1 & 0 \\0 &z_{2}
\end{smallmatrix}\right).$ Then, using \cite[Equation 2.10]{pal1} we can  write $\mathcal G_{F(\lambda)}\left(\left(\begin{smallmatrix}z_1 & 0 \\0 &z_{2}
\end{smallmatrix}\right)\right)$ as \begin{align}\label{mathcalGG}
		\mathcal G_{F(\lambda)}\left(\left(\begin{smallmatrix}z_1 & 0 \\0 &z_{2}
		\end{smallmatrix}\right)\right) &=F_{11}(\lambda)+\left(\begin{smallmatrix}F_{12}(\lambda)&F_{13}(\lambda)
		\end{smallmatrix}\right)\left(\begin{smallmatrix}z_1 & 0 \\0 &z_{2}
		\end{smallmatrix}\right)\left(I_{2}-\left(\begin{smallmatrix}F_{22}(\lambda)&F_{23}(\lambda)\\F_{32} (\lambda)& F_{33}(\lambda)
		\end{smallmatrix}\right)\left(\begin{smallmatrix}z_1 & 0 \\0 &z_{2}
		\end{smallmatrix}\right)\right)^{-1}\left(\begin{smallmatrix}F_{21}(\lambda)\\F_{31}(\lambda)
		\end{smallmatrix}\right)
\end{align} for all $z_1,z_2\in\mathbb C$ such that $\det(C(z_1,z_2,\lambda))\neq 0,$ where $C(z_1,z_2,\lambda)=\left(I_{2}-\left(\begin{smallmatrix}F_{22} (\lambda)& F_{23}(\lambda) \\F_{32}(\lambda) &F_{33}(\lambda)
	\end{smallmatrix}\right)\left(\begin{smallmatrix}z_1 & 0 \\0 &z_{2}
	\end{smallmatrix}\right)\right).$
From \cite[Equation 2.14 and Equation 2.15]{pal1} we have
\begin{align}\label{gaa}\tilde{\gamma}\left(\left(\begin{smallmatrix}z_1 & 0 \\0 &z_{2}
	\end{smallmatrix}\right), \lambda \right)\nonumber&=\left(I_{2}-\left(\begin{smallmatrix}F_{22} (\lambda)& F_{23}(\lambda) \\F_{32} (\lambda)&F_{33}(\lambda)
	\end{smallmatrix}\right)\left(\begin{smallmatrix}z_1 & 0 \\0 &z_{2}
	\end{smallmatrix}\right)\right)^{-1}\left(\begin{smallmatrix}F_{21} (\lambda)\\ F_{31}(\lambda)
	\end{smallmatrix}\right)\\&=\left(\begin{smallmatrix}\tilde{\gamma}_1(z_1,z_2,\lambda )\\ \tilde{\gamma}_2(z_1,z_2,\lambda)
	\end{smallmatrix}\right)
\end{align} 
and
\begin{align}\label{ett}\tilde{\eta}\left(\left(\begin{smallmatrix}z_1 & 0 \\0 &z_{2}
	\end{smallmatrix}\right),\lambda \right)\nonumber&=\left(\begin{smallmatrix}1\\\left(\begin{smallmatrix}z_1 & 0 \\0 &z_{2}
		\end{smallmatrix}\right)\tilde{\gamma}\left(\left(\begin{smallmatrix}z_1 & 0 \\0 &z_{2}
		\end{smallmatrix}\right),\lambda \right)
	\end{smallmatrix}\right)\\&=\left(\begin{smallmatrix}1\\z_1\tilde{\gamma}_1(z_1,z_2,\lambda) \\ z_2\tilde{\gamma}_2(z_1,z_2,\lambda)
	\end{smallmatrix}\right)
\end{align}
for all $z_1,z_2\in\mathbb C$ such that $\det(C(z_1,z_2,\lambda))\neq 0,$ where $$\tilde{\gamma}_1(z_1,z_2,\lambda)=\frac{(1-F_{33}(\lambda)z_2)F_{21}(\lambda)+z_2F_{23}(\lambda)F_{31}(\lambda)} {\det(C(z_1,z_2,\lambda))}$$ and $$\tilde{\gamma}_2(z_1,z_2,\lambda)=\frac{(1-F_{22}(\lambda)z_1)F_{31}(\lambda)+z_1F_{32}(\lambda)F_{21}(\lambda)} {\det(C(z_1,z_2,\lambda))}.$$
We notice from (\cite[Equation 2.25]{pal1}) that $\mathcal G_{F(\lambda)}\left(\left(\begin{smallmatrix}z_1 & 0 \\0 &z_{2}
\end{smallmatrix}\right)\right) $ is a rational function in $z_1,z_2.$ In this section we use the notation $\mathcal G_{F(\lambda)}(z_1,z_2) $ instead of $\mathcal G_{F(\lambda)}\left(\left(\begin{smallmatrix}z_1 & 0 \\0 &z_{2}
\end{smallmatrix}\right)\right) .$

	\begin{defn}
		The map $SE:\mathcal S_{1}(\mathbb C^3,\mathbb C^3)\to \mathcal S_{3}(\mathbb C,\mathbb C)$ defined as follows: $$SE(F)(z_1,z_2,\lambda)=-\mathcal{G}_{F(\lambda)}(z_1,z_2),~ \lambda \in \mathbb{D},~(z_1,z_2)\in \mathbb{D}^{2}$$	
	\end{defn}
	We show that the map $SE$ is well defined. In order to prove this, it suffices to show that $|\mathcal{G}_{F(\lambda)}(z_1,z_2)|\leq 1,~ {\rm{for ~all}}~\lambda \in \mathbb{D},~(z_1,z_2)\in \mathbb{D}^{2}.$ By characterization of $ \Gamma_{E(2;2;1,1)},$ we deduce that if $\det(C(z_1,z_2,\lambda))\neq 0$ for all $(z_1,z_2)\in \mathbb{D}^{2}$  and for all $\lambda\in \mathbb D$ if and only if $$(F_{22}(\lambda),F_{33}(\lambda), \det A(\lambda))\in \Gamma_{E(2;2;1,1)},$$ where $A(\lambda)=\left(\begin{smallmatrix}F_{22} (\lambda)& F_{23}(\lambda) \\F_{32} (\lambda)&F_{33}(\lambda)
	\end{smallmatrix}\right).$ 
	From [Proposition $2.11$, \cite{pal1}], we get  $\sup_{(z_1,z_2)\in\mathbb D^2}|\mathcal{G}_{F(\lambda)}(z_1,z_2)|\leq 1$ and hence $|\mathcal{G}_{F(\lambda)}(z_1,z_2)|\leq 1$ for all $(z_1,z_2)\in \mathbb D^2.$
	The following proposition is evident from above discussion.
	\begin{prop}
		The map $SE$ is well defined.
	\end{prop}
	
	\begin{rem}
		It follows from  \eqref{mathcalGG} that if either $\left(\begin{array}{cc}
			F_{21}(\lambda)\\
			F_{31}(\lambda)
		\end{array}\right)=\bold{0}$ or $\left(F_{12}(\lambda), F_{13}(\lambda)\right)=\bold{0}$, then the map $$SE(F)(z_1,z_2,\lambda))=-\mathcal{G}_{F(\lambda)}(z_1,z_2)=-F_{11}(\lambda)$$  which is independent of $z_1,z_2.$ In particular, if $z_1=z_2$, then $$SE(F)(z_1,z_1,\lambda)=-\mathcal{G}_{F(\lambda)}(z_1),~ \lambda \in \mathbb{D},~z_1\in \mathbb{D}.$$	For this case $SE$ maps from $\mathcal S_{1}(\mathbb C^3,\mathbb C^3)$ to $\mathcal S_{2}(\mathbb C,\mathbb C).$
	\end{rem} 
	\section{Relation between $\mathcal{S}_{1}(\mathbb C^3,\mathbb C^3)$ and the set of analytic kernels on $\mathbb{D}^{3}$}
	In this section we discuss the relation between  $\mathcal{S}_{1}(\mathbb C^3,\mathbb C^3)$  and the set of analytic kernels on $\mathbb{D}^{3}$. We recall the definition of analytic kernel from \cite{AglerM, aron}. Let $\mathcal F $ be a class of functions  defined on a set  $E,$  which form a real or complex Hilbert function space.  We refer to the function $K(x, y) $ of $x$ and $y $ in $E$ as a reproducing kernel of $\mathcal F$ if it meets the following properties: 
	\begin{enumerate}
		\item   for every $y\in E$, $K(x, y) $ is a function of $x$ belongs to $\mathcal F;$
		\item it satisfies the reproducing property, that is, for every $y\in E$ and for every $f\in\mathcal F$,
		$$\langle f, K(x,y)\rangle=f(y).$$
	\end{enumerate} 
	Let  $N^{(i)}$, for $i=1,2,3$ represent the analytic kernels on $\mathbb{D}^{3}$. For $\textbf{z}=(z_1,z_2)$ and $\textbf{w}=(w_1,w_2)$, assume that $K_{(N^{(1)},N^{(2)},N^{(3)}
		)}$ is the hermitian symmetric function on $\mathbb{D}^{3}\times \mathbb{D}^{3}$ given by 
	\begin{align}\label{kernel}
		K_{(N^{(1)},N^{(2)},N^{(3)})}(\textbf{z},\lambda,\textbf{w},\mu)\nonumber&=1-(1-\bar{w_{1}}z_{1})N^{(1)}(\textbf{z},\lambda,\textbf{w},\mu)-(1-\bar{w_{2}}z_{2})N^{(2)} (\textbf{z},\lambda,\textbf{w},\mu)\\
		&-(1-\bar{\mu}\lambda)N^{(3)}(\textbf{z},\lambda,\textbf{w},\mu).
	\end{align}
	In particular for $w_1=w_2$ and $z_1=z_2$ from \eqref{kernel}, we have 
	\begin{align}\label{kernel1}
		K_{(N^{(1)},N^{(2)},N^{(3)})}(z_1,\lambda,w_1,\mu)\nonumber&=1-(1-\bar{w_{1}}z_{1})(N^{(1)}(z_1,\lambda,w_1,\mu)+N^{(2)} (z_1,\lambda,w_1,\mu))\\
		&-(1-\bar{\mu}\lambda)N^{(3)}(z_1,\lambda,w_1,\mu).
	\end{align}
	
	The set $\tilde{\mathcal{R}_{1}}$ and $\tilde{\mathcal{S}_{1}}$ are defined as follows: $$\begin{aligned}
		\tilde{\mathcal{R}_{1}}:=&\{(N^{(1)},N^{(2)},N^{(3)}):N^{(1)},N^{(2)},N^{(3)},K_{(N^{(1)},N^{(2)},N^{(3)})}~\text{are the analytic kernel on }\\ &\mathbb{D}^{3} ~\text{and} \,\,K_{(N^{(1)},N^{(2)},N^{(3)})}\,\, \text{is of rank 1}\}.
	\end{aligned}$$	
	
	and $$\begin{aligned}
		\tilde{\mathcal{S}_{1}}:=&\{(N^{(1)}+N^{(2)},N^{(3)}):N^{(1)},N^{(2)},N^{(3)},K_{(N^{(1)},N^{(2)},N^{(3)})}~\text{are the analytic kernel on }\\ &\mathbb{D}^{2} ~\text{and} \,\,K_{(N^{(1)},N^{(2)},N^{(3)})}\,\, \text{is of rank 1}\}
	\end{aligned}$$	 respectively.

	\subsection{The map Upper~$E$}
	Note that for every $F=[F_{ij}]_{i,j=1}^{3}\in \mathcal{S}_{1}(\mathbb C^3,\mathbb C^3)$  from \eqref{gaa} and \eqref{ett}, we have $$\tilde{\gamma}(\lambda,\textbf{z})=\left(\begin{array}{cc}
		\tilde{\gamma}_{1}(\lambda,\textbf{z})\\
		\tilde{\gamma}_{2}(\lambda,\textbf{z})
	\end{array}\right)~{\rm{and}}~\tilde{\eta}(\lambda,\textbf{z})=\left(\begin{array}{ccc}
		1\\
		z_{1}\tilde{\gamma}_{1}(\lambda,\textbf{z})\\
		z_{2}\tilde{\gamma}_{2}(\lambda,\textbf{z})
	\end{array}\right).$$ 
	The functions $N^{(i)}_{F}$  for $i=1,2,3$ on $\mathbb{D}^{3}\times\mathbb{D}^{3}$ are described below
	\begin{equation}\label{eq8}
		N_{F}^{(i)}(\lambda,\textbf{z},\mu,\textbf{w})=\overline{\tilde{\gamma}_{i}(\mu,\textbf{w})}\tilde{\gamma}_{i}(\lambda,\textbf{z}), \,\, i=1,2
	\end{equation} and
	\begin{equation}\label{eq9}
		N^{(3)}_{F}(\lambda,\textbf{z},\mu,\textbf{w})=\tilde{\eta}(\mu,\textbf{w})^{*}\frac{I-F(\mu)^{*}F(\lambda)}{1-\bar{\mu}\lambda}\tilde{\eta}(\lambda,\textbf{z})
	\end{equation} for all $\lambda,\mu \in \mathbb{D}\,\, \text{and} \,\,\textbf{z},\textbf{w}\in \mathbb{D}^{2}$. As $\det(C(z_1,z_2,\lambda)\neq 0$, the functions $N^{(i)}_{F}$ for $i=1,2,3,$ are well-defined.
	\begin{prop}\label{upper}
		
		Suppose that $F \in  \mathcal{S}_{1}(\mathbb C^3,\mathbb C^3)$ with $\left(\begin{array}{cc}
			F_{21}(\lambda)\\
			F_{31}(\lambda)
		\end{array}\right)\neq \bold{0}.$ Then the maps $N^{(i)}_{F}$ for $i=1,2,3,$ are the analytic kernels on $\mathbb{D}^{3}$, $N^{(i)}_{F},~i=1,2,$ are of rank 1 and $(N_{F}^{(1)},N_{F}^{(2)},N_{F}^{(3)})\in \mathcal{\tilde{R}}_{1}$.
	\end{prop}
	\begin{proof}
		By definition of $SE$, the maps $N^{(i)}_{F}$ for $i=1,2,3,$ are analytic on $\mathbb{D}^{3}$.  As $\left(\begin{array}{cc}
			F_{21}(\lambda)\\
			F_{31}(\lambda)
		\end{array}\right)\neq \bold{0}$, it follows from  \eqref{eq8} that $\tilde{\gamma}_i:\mathbb D^3\to \mathbb C$ for $i=1,2$ is not equal to zero and hence $N_{F}^{(i)},i=1,2$ are analytic kernels on $\mathbb{D}^{3}$ of rank $1. $
		
		We now wish to demonstrate that $(N_{F}^{(1)},N_{F}^{(2)},N_{F}^{(3)})\in \mathcal{\tilde{R}}_{1}$. In order to see this, it is sufficient to prove that $K_{(N_F^{(1)},N_F^{(2)},N_F^{(3)})}$ is an analytic kernel on $\mathbb{D}^{3}$ of rank $1.$
		For all $\lambda,\mu \in \mathbb{D}\,\, \text{and} \,\,\textbf{z},\textbf{w}\in \mathbb{D}^{2}$, it yields from \cite[Proposition 2.11]{pal1} that
		\begin{align}\label{eq11}
			1-\overline{\mathcal G_{F(\mu)}(\textbf{w})}\mathcal G_{F(\lambda)}(\textbf{z})\nonumber&=(1-\bar{w_{1}}z_{1})N_F^{(1)}(\textbf{z},\lambda,\textbf{w},\mu)-(1-\bar{w_{2}}z_{2})N_F^{(2)}(\textbf{z},\lambda,\textbf{w},\mu)\\&-(1-\bar{\mu}\lambda)N_F^{(3)}(\textbf{z},\lambda,\textbf{w},\mu).
		\end{align}
		We observe from \eqref{kernel} and \eqref{eq11} that
		$$K_{(N_F^{(1)},N_F^{(2)},N_F^{(3)})}(\textbf{z},\lambda,\textbf{w},\mu)=\overline{\mathcal G_{F(\mu)}(\textbf{w})}\mathcal G_{F(\lambda)}(\textbf{z}).$$ This shows that $K_{(N_F^{(1)},N_F^{(2)},N_F^{(3)})}$ 
		is a kernel with rank $1$ and hence  $(N_{F}^{(1)},N_{F}^{(2)},N_{F}^{(3)})\in \mathcal{\tilde{R}}_{1}$.
		
	\end{proof}	
	The proof of the following proposition is same as Proposition \eqref{upper} and hence we omit the proof.
	\begin{prop}\label{upper11}
		
		Let $F \in  \mathcal{S}_{1}(\mathbb C^3,\mathbb C^3)$ with $\left(\begin{array}{cc}
			F_{21}(\lambda)\\
			F_{31}(\lambda)
		\end{array}\right)\neq \bold{0}.$ Then the maps $N^{(i)}_{F}$ for \\$i=1,2,3,$ are the analytic kernels on $\mathbb{D}^{2}$, $N^{(i)}_{F},~i=1,2,$ are of rank 1 and $(N_{F}^{(1)}+N_{F}^{(2)},N_{F}^{(3)})\in \mathcal{\tilde{S}}_{1}$.
	\end{prop}	
	\begin{prop}\label{upper1}
		Let $F \in  \mathcal{S}_{1}(\mathbb C^3,\mathbb C^3)$ with $\left(\begin{array}{cc}
			F_{21}(\lambda)\\
			F_{31}(\lambda)
		\end{array}\right)= \bold{0}.$  Then the maps $N_{F}^{(i)},i=1,2,3,$ \\are the analytic kernels on $\mathbb{D}^{3}$, $N_{F}^{(i)}, i=1,2$ is of rank 0 and $(N_{F}^{(1)},N_{F}^{(2)},N_{F}^{(3)})\in \mathcal{\tilde{R}}_{1}$. Furthermore, $$N_{F}^{(i)}(\lambda,\textbf{z},\mu,\textbf{w})=0\,\, \text{for}\,\,i=1,2, \,\,N_{F}^{(3)}(\lambda,\textbf{z},\mu,\textbf{w})=\frac{1-F_{11}(\mu)^{*}F_{11}(\lambda)}{1-\bar{\mu}\lambda}$$ and $$K_{(N_{F}^{(1)},N_{F}^{(2)},N_F^{(3)})}(\textbf{z},\lambda,\textbf{w},\mu)=F_{11}(\mu)^{*}F_{11}(\lambda) ~{\rm{for~ all~}}\lambda,\mu \in \mathbb{D}\,\, \text{and} \,\,\textbf{z},\textbf{w}\in \mathbb{D}^{2}.$$
	\end{prop}
	\begin{proof}
		For every $F \in  \mathcal{S}_{1}(\mathbb C^3,\mathbb C^3)$ with $\left(\begin{array}{cc}
			F_{21}(\lambda)\\
			F_{31}(\lambda)
		\end{array}\right)= \bold{0},$ from \eqref{gaa} and \eqref{ett} we have
		$$\tilde{\gamma}(\lambda,\textbf{z})=\left(\begin{array}{cc}
			0 \\
			0
		\end{array}\right)~\text{and}~\tilde{\eta}(\lambda,\textbf{z})=\left(\begin{array}{ccc}
			1 \\
			0 \\
			0
		\end{array}\right)~{\rm{for ~all }}\lambda \in \mathbb{D},\textbf{z}\in \mathbb{D}^{2}.$$ Hence by definition of $N_{F}^{(i)}(\lambda,\textbf{z},\mu,\textbf{w})$, it implies that $N_{F}^{(i)}(\lambda,\textbf{z},\mu,\textbf{w})=0,\,\,i=1,2.$  Clearly, for all $\lambda,\mu \in \mathbb{D}\,\, \text{and} \,\,\textbf{z},\textbf{w}\in \mathbb{D}^{2}$, $N_{F}^{(i)}(\lambda,\textbf{z},\mu,\textbf{w})$ for $i=1,2,$ are analytic kernel on $\mathbb{D}^3$ with rank 0. Note that
		$$N_{F}^{(3)}(\lambda,\textbf{z},\mu,\textbf{w})=\frac{1-F_{11}(\mu)^{*}F_{11}(\lambda)}{1-\bar{\mu}\lambda}$$ which indicates that it is independent of $\textbf{z},\textbf{w}\in \mathbb D^2$ and hence  $N_F^{(3)}$ is an analytic kernel on $\mathbb{D}^{3}$.  We observe that
		$$K_{(N_{F}^{(1)},N_{F}^{(2)},N_F^{(3)})}(\textbf{z},\lambda,\textbf{w},\mu)=1-(1-\bar{\mu}\lambda)N_{F}^{(3)}(\lambda,\textbf{z},\mu,\textbf{w})=\overline{F_{11}(\mu)}F_{11}(\lambda)$$ which is again independent of $\textbf{z},\textbf{w}\in \mathbb D^2. $ This shows that $K_{(N_{F}^{(1)},N_{F}^{(2)},N_{F}^{(3)})}$ is an analytic kernel of rank $1$ on $\mathbb{D}^3$ and hence $(N_{F}^{(1)},N_{F}^{(2)},N_{F}^{(3)})\in \mathcal{\tilde{R}}_{1}$. 
	\end{proof}
	The proof of the following 	proposition is same as the Proposition \eqref{upper1}. Therefore, we skip the proof.
	\begin{prop}\label{upper12}
		Assume that $F \in  \mathcal{S}_{1}(\mathbb C^3,\mathbb C^3)$ with $\left(\begin{array}{cc}
			F_{21}(\lambda)\\
			F_{31}(\lambda)
		\end{array}\right)= \bold{0}.$  Then the maps $N_{F}^{(i)},i=1,2,3,$ are the analytic kernels on $\mathbb{D}^{2}$, $N_{F}^{(i)}, i=1,2$ is of rank 0 and $(N_{F}^{(1)}+N_{F}^{(2)},N_{F}^{(3)})\in \mathcal{\tilde{S}}_{1}$. Furthermore, $$N_{F}^{(i)}(\lambda,z,\mu,w)=0\,\, \text{for}\,\,i=1,2, \,\,N_{F}^{(3)}(\lambda,z,\mu,w)=\frac{1-F_{11}(\mu)^{*}F_{11}(\lambda)}{1-\bar{\mu}\lambda}$$ and $$K_{(N_{F}^{(1)},N_{F}^{(2)},N_F^{(3)})}(z,\lambda,w,\mu)=F_{11}(\mu)^{*}F_{11}(\lambda) ~{\rm{for~ all~}}\lambda,\mu \in \mathbb{D}\,\, \text{and} \,\,z,w\in \mathbb{D}.$$
	\end{prop}
	
	\begin{defn}
		For every	$F \in  \mathcal{S}_{1}(\mathbb C^3,\mathbb C^3),$ the upper map $Upper E:\mathcal{S}_{1}(\mathbb C^3,\mathbb C^3)\rightarrow\tilde{\mathcal{R}_{1}}$ is given by $$Upper E(F)=(N_{F}^{(1)},N_{F}^{(2)},N_{F}^{(3)}).$$
	\end{defn}
	Note that from Proposition \eqref{upper} and \eqref{upper1} the  map Upper$E$ is well defined.
	\begin{defn}
		For every	$F \in  \mathcal{S}_{1}(\mathbb C^3,\mathbb C^3),$ the upper map $Upper E:\mathcal{S}_{1}(\mathbb C^3,\mathbb C^3)\rightarrow\tilde{\mathcal{S}_{1}}$ is given by $$Upper E(F)=(N_{F}^{(1)}+N_{F}^{(2)},N_{F}^{(3)}).$$
	\end{defn}
	Note that from Proposition \eqref{upper11} and \eqref{upper12} the  map Upper$E$ is well defined.
	\section{$UW$ procedure via the set valued maps Upper$W:\mathcal{\tilde{R}}_{11}\to \mathcal S_{1}(\mathbb C^3,\mathbb C^3)$ and Upper$W:\mathcal{\tilde{S}}_{11}\to \mathcal S_{1}(\mathbb C^3,\mathbb C^3)$}
	In this section we describe $UW$ procedure. This procedure  ensures the well definedness of the maps Upper$W:\mathcal{\tilde{R}}_{11}\to \mathcal S_{1}(\mathbb C^3,\mathbb C^3)$  and Upper$W:\mathcal{\tilde{S}}_{11}\to \mathcal S_{1}(\mathbb C^3,\mathbb C^3),$ where  $$\begin{aligned}
		\tilde{\mathcal{R}}_{11}=&\{(N^{(1)},N^{(2)},N^{(3)})~\text{are the analytic kernel on }\mathbb{D}^{3}\\&\,\, \text{and} \,\,N^{(1)},N^{(2)},K_{(N^{(1)},N^{(2)},N^{(3)})}\,\, \text{are of rank 1}\}\subseteq \tilde{\mathcal{R}_1}
	\end{aligned}$$	
	and 
	$$\begin{aligned}
		\tilde{\mathcal{S}}_{11}=&\{(N^{(1)}+N^{(2)},N^{(3)})~\text{are the analytic kernel on }\mathbb{D}^{2}\\&\,\, \text{and} \,\,N^{(1)},N^{(2)},K_{(N^{(1)},N^{(2)},N^{(3)})}\,\, \text{are of rank 1}\}\subseteq \tilde{\mathcal{S}_1}
	\end{aligned}$$	
	It is well-known \cite{AglerM} that every Hilbert function space on a set $X$ produces a kernel $K$ on a set $X$ and for each kernel $K$, one can construct a Hilbert function space $\mathcal H_{K}$ that has $K$ as its reproducing kernel.
	Suppose $F \in \mathcal S_{1}(\mathbb C^3,\mathbb C^3)$  with $\left(\begin{array}{cc}
		F_{21}(\lambda)\\
		F_{31}(\lambda)
	\end{array}\right)\neq \bold{0}.$  Then from Proposition \eqref{kernel}, we observe that the kernels $N_F^{(i)},i=1,2$ have rank one. In this case we conclude that the Upper$E$ maps into a subset $\mathcal{\tilde{R}}_{11}$ of $\mathcal{\tilde{R}}_{1}$  rather than whole $\mathcal{\tilde{R}}_{1}$ . The following theorem tells the $UW$ procedure. We use this procedure later to define Upper$W$ maps from $\mathcal{\tilde{R}}_{11}$ to $\mathcal S_{1}(\mathbb C^3,\mathbb C^3).$

	\begin{thm}\label{UW}
		(Procedure UW) Suppose that  $(N^{(1)},N^{(2)},N^{(3)})\in \tilde{R}_{11}.$ Then there are functions $f_{1}\in \mathcal{H}_{N^{(1)}}, f_{2}\in \mathcal{H}_{N^{(2)}}$ and $g\in H_{K_{(N^{(1)},N^{(2)},N^{(3)})}}$ so that 
		$$N^{(i)}(\lambda,\textbf{z},\mu,\textbf{w})=\overline{f_{i}(\mu,\textbf{w})}f_{i}(\lambda,\textbf{z}), i=1,2$$
		and 
		$$K_{(N^{(1)},N^{(2)},N^{(3)})}(\lambda,\textbf{z},\mu,\textbf{w})=\overline{g(\mu,\textbf{w})}g(\lambda,\textbf{z})$$ for each $\lambda, \mu\in \mathbb{D}$, $\textbf{z}=(z_{1},z_{2}),\textbf{w}=(w_{1},w_{2})\in \mathbb{D}^{2}$ and a function $\Xi\in \mathcal S_{1}(\mathbb C^3,\mathbb C^3)$ so that $$\Xi\left(\begin{array}{ccc}
			1\\ z_{1}f_{1}(\lambda,\textbf{z}) \\ z_{2}f_{2}(\lambda,\textbf{z})
		\end{array}\right)=\left(\begin{array}{ccc}
			g(\lambda,\textbf{z})\\ f_{1}(\lambda,\textbf{z}) \\ f_{2}(\lambda,\textbf{z})
		\end{array}\right),$$
		$\text{for each}\,\,\lambda\in \mathbb{D}, \textbf{z}=(z_{1},z_{2})\in \mathbb{D}^{2}$.
	\end{thm}
	\begin{proof}
		We will construct the function $\Xi$ by using the famous Lurking isometry argument. It is given that $(N^{(1)},N^{(2)},N^{(3)})\in \tilde{R}_{11},$ then there exist $f_{1}\in \mathcal{H}_{N^{(1)}}, f_{2}\in \mathcal{H}_{N^{(2)}}$, $v_{\lambda,\textbf{z}}\in  \mathcal{H}_{N^{(3)}}$ and $g\in H_{K_{(N^{(1)},N^{(2)},N^{(3)})}}$ such that 
		$$N^{(i)}(\lambda,\textbf{z},\mu,\textbf{w})=\overline{f_{i}(\mu,\textbf{w})}f_{i}(\lambda,\textbf{z}), i=1,2,$$
		$$K_{(N^{(1)},N^{(2)},N^{(3)})}(\lambda,\textbf{z},\mu,\textbf{w})=\overline{g(\mu,\textbf{w})}g(\lambda,\textbf{z})$$
		and $$N^{(3)}(\lambda,\textbf{z},\mu,\textbf{w})=\left\langle v_{\lambda,\textbf{z}}^1, v_{\mu,\textbf{w}}^1\right\rangle_{\mathcal{H}_{N^{(3)}}}$$
		for all $\lambda,\mu \in \mathbb{D}\,\, \text{and} \,\,\textbf{z},\textbf{w}\in \mathbb{D}^{2}$. From equation \eqref{kernel} we have 
		\begin{equation}\label{eq12}
			\begin{aligned} \overline{g(\mu,\textbf{w})}g(\lambda,\textbf{z})&=1-(1-\bar{w}_1z_1)\overline{f_{1}(\mu,\textbf{w})}f_{1}(\lambda,\textbf{z})-(1-\bar{w}_2z_2)\overline{f_{2}(\mu,\textbf{w})}f_{2}(\lambda,\textbf{z})\\&-(1-\bar{\mu}\lambda)\langle v_{\lambda,\textbf{z}}, v_{\mu,\textbf{w}}\rangle_{\mathcal{H}_{N^{(3)}}}
			\end{aligned}	
		\end{equation}
		Rearranging the equation \eqref{eq12} we get, 
		\begin{equation}
			\begin{aligned}
				\overline{g(\mu,\textbf{w})}g(\lambda,\textbf{z})+\overline{f_{1}(\mu,\textbf{w})}f_{1}(\lambda,\textbf{z})+\overline{f_{2}(\mu,\textbf{w})}f_{2}(\lambda,\textbf{z})+\left\langle v_{\lambda,\textbf{z}}^1, v_{\mu,\textbf{w}}^1\right\rangle_{\mathcal{H}_{N^{(3)}}}&
				\\=1+\bar{w}_1z_1\overline{f_{1}\mu,\textbf{w})}f_{1}(\lambda,\textbf{z})+\bar{w}_2z_2\overline{f_{2}(\mu,\textbf{w})}f_{2}(\lambda,\textbf{z})+\bar{\mu}\lambda\left\langle v_{\lambda,\textbf{z}}^1, v_{\mu,\textbf{w}}^1\right\rangle_{\mathcal{H}_{N^{(3)}}},
			\end{aligned}
		\end{equation}
		for all $\lambda,\mu \in \mathbb{D}\,\, \text{and} \,\,\textbf{z},\textbf{w}\in \mathbb{D}^{2}$. The above equation can also be written in the form
		\begin{equation}\label{eq17}
			\left\langle\left(\begin{array}{cccc}
				g(\lambda,\textbf{z})\\ f_{1}(\lambda,\textbf{z}) \\ f_{2}(\lambda,\textbf{z}) \\ v_{\lambda,\textbf{z}} 
			\end{array}\right),\left(\begin{array}{cccc}
				g(\mu,\textbf{w})\\ f_{1}(\mu,\textbf{w}) \\ f_{2}(\mu,\textbf{w}) \\ v_{\mu,\textbf{w}} 
			\end{array}\right)\right\rangle_{\mathbb{C}^3\oplus \mathcal{H}_{N^{(3)}}}=\left\langle\left(\begin{array}{cccc}
				1\\ z_{1}f_{1}(\lambda,\textbf{z}) \\ z_{2}f_{2}(\lambda,\textbf{z}) \\ \lambda v_{\lambda,\textbf{z}}
			\end{array}\right),\left(\begin{array}{cccc}
				1\\ z_{1}f_{1}(\mu,\textbf{w}) \\ z_{2}f_{2}(\mu,\textbf{w}) \\ \mu v_{\mu,\textbf{w}}
			\end{array}\right)\right\rangle_{\mathbb{C}^3\oplus \mathcal{H}_{N^{(3)}}},
		\end{equation}
		for all $\lambda,\mu \in \mathbb{D}\,\, \text{and} \,\,\textbf{z},\textbf{w}\in \mathbb{D}^{2}$. 
		From \eqref{eq17} we observe that the Gramian of vectors
		$$\left(\begin{array}{cccc}
			g(\lambda,\textbf{z})\\ f_{1}(\lambda,\textbf{z}) \\ f_{2}(\lambda,\textbf{z}) \\ v_{\lambda,\textbf{z}} 
		\end{array}\right) \in \mathbb{C}^3\oplus \mathcal{H}_{N^{(3)}}~\text{for all}~\lambda,\mu \in \mathbb{D}\,\, \text{and} \,\,\textbf{z},\textbf{w}\in \mathbb{D}^{2} $$ is equal to the Gramian of vectors
		$$\left(\begin{array}{cccc}
			1\\ z_{1}f_{1}(\lambda,\textbf{z}) \\ z_{2}f_{2}(\lambda,\textbf{z}) \\ \lambda v_{\lambda,\textbf{z}}
		\end{array}\right)
		\in \mathbb{C}^3\oplus \mathcal{H}_{N^{(3)}}~\text{for all}~\lambda,\mu \in \mathbb{D}\,\, \text{and} \,\,\textbf{z},\textbf{w}\in \mathbb{D}^{2}. 
		$$
		Therefore, there exist an isometry 
		$$V_{0}:\operatorname{span}\left\{\left(\begin{array}{cccc}
			1\\ z_{1}f_{1}(\lambda,\textbf{z}) \\ z_{2}f_{2}(\lambda,\textbf{z}) \\ \lambda v_{\lambda,\textbf{z}}
		\end{array}\right)
		\in \mathbb{C}^3\oplus \mathcal{H}_{N^{(3)}}~\text{for all}~\lambda,\mu \in \mathbb{D}\,\, \text{and} \,\,\textbf{z},\textbf{w}\in \mathbb{D}^{2}\right\}\rightarrow \mathbb{C}^3\oplus \mathcal{H}_{N^{(3)}}$$ 
		such that 
		$$V_{0}\left(\begin{array}{cccc}
			1\\ z_{1}f_{1}(\lambda,\textbf{z}) \\ z_{2}f_{2}(\lambda,\textbf{z}) \\ \lambda v_{\lambda,\textbf{z}}
		\end{array}\right)=\left(\begin{array}{cccc}
			g(\lambda,\textbf{z})\\ f_{1}(\lambda,\textbf{z}) \\ f_{2}(\lambda,\textbf{z}) \\ v_{\lambda,\textbf{z}} 
		\end{array}\right)$$
		for all $\lambda \in \mathbb{D}\,\, \text{and} \,\,\textbf{z}\in \mathbb{D}^{2}.$
		We extend $V_0$ to a contraction $V$ on $(\mathbb{C}^{3}\oplus \mathcal{H}_{N^{(3)}} )$ by defining $ V$ to be $0$ on $$(\mathbb{C}^{3}\oplus \mathcal{H}_{N^{(3)}} )~\ominus~\operatorname{span}\{(1, z_{1}f_{1}(\lambda,\textbf{z}), z_{2}f_{2}(\lambda,\textbf{z}), \lambda v_{\lambda,\textbf{z}}
		)\}$$ and express $V$ as a block operator matrix $$V=\begin{bmatrix}
			P & Q \\ R &S
		\end{bmatrix}:\mathbb{C}^{3}\oplus \mathcal{H}_{N^{(3)}}\rightarrow\mathbb{C}^{3}\oplus \mathcal{H}_{N^{(3)}}$$ where $P:\mathbb{C}^{3}\rightarrow\mathbb{C}^{3},~Q:\mathcal{H}_{N^{(3)}}\rightarrow\mathbb{C}^{3},~R:\mathbb{C}^{3}\rightarrow\mathcal{H}_{N^{(3)}}$ and $S:\mathcal{H}_{N^{(3)}}\rightarrow \mathcal{H}_{N^{(3)}}$. Then $V$ satisfies the expression $$\begin{bmatrix}
			P & Q \\ R &S
		\end{bmatrix}\left(\begin{array}{cccc}
			1\\ z_{1}f_{1}(\lambda,\textbf{z}) \\ z_{2}f_{2}(\lambda,\textbf{z}) \\ \lambda v_{\lambda,\textbf{z}}
		\end{array}\right)=\left(\begin{array}{cccc}
			g(\lambda,\textbf{z})\\ f_{1}(\lambda,\textbf{z}) \\ f_{2}(\lambda,\textbf{z}) \\ v_{\lambda,\textbf{z}} 
		\end{array}\right)$$ for all $\lambda \in \mathbb{D}\,\, \text{and} \,\,\textbf{z}\in \mathbb{D}^{2}$. Solving the above equation we get 
		\begin{equation}\label{eq15}
			P\left(\begin{array}{ccc}
				1\\ z_{1}f_{1}(\lambda,\textbf{z}) \\ z_{2}f_{2}(\lambda,\textbf{z})
			\end{array}\right)+Q\lambda v_{\lambda,\textbf{z}}=\left(\begin{array}{cccc}
				g(\lambda,\textbf{z})\\ f_{1}(\lambda,\textbf{z}) \\ f_{2}(\lambda,\textbf{z}) 
			\end{array}\right)	
		\end{equation}
		and 
		\begin{equation}\label{eq16}
			R\left(\begin{array}{ccc}
				1\\ z_{1}f_{1}(\lambda,\textbf{z}) \\ z_{2}f_{2}(\lambda,\textbf{z})
			\end{array}\right)+S\lambda v_{\lambda,\textbf{z}}= v_{\lambda,\textbf{z}}	
		\end{equation}
		for all $\lambda \in \mathbb{D}\,\, \text{and} \,\,\textbf{z}\in \mathbb{D}^{2}$.
		From \eqref{eq16} we have $$v_{\lambda,\textbf{z}}=(I_{\mathcal{H}_{N^{(3)}}}-S\lambda)^{-1}R\left(\begin{array}{ccc}
			1\\ z_{1}f_{1}(\lambda,\textbf{z}) \\ z_{2}f_{2}(\lambda,\textbf{z})
		\end{array}\right)$$ for all $\lambda \in \mathbb{D}\,\, \text{and} \,\,\textbf{z}\in \mathbb{D}^{2}$ Putting the value of $v_{\lambda,\textbf{w}}$ in equation \eqref{eq15} we have the form $$(P+Q\lambda(I_{\mathcal{H}_{N^{(3)}}}-S\lambda)^{-1}R)\left(\begin{array}{ccc}
			1\\ z_{1}f_{1}(\lambda,\textbf{z}) \\ z_{2}f_{2}(\lambda,\textbf{z})
		\end{array}\right)=\left(\begin{array}{ccc}
			g(\lambda,\textbf{z})\\ f_{1}(\lambda,\textbf{z}) \\ f_{2}(\lambda,\textbf{z}) 
		\end{array}\right).$$ Let $\mathcal{G}_{V}(\lambda)=P+Q\lambda(I_{\mathcal{H}_{N^{(3)}}}-S\lambda)^{-1}R,$ then from \cite[Proposition 2.8]{pal1} and \cite[Proposition 2.9]{pal1} we have $\|\mathcal{G}_{V}(\lambda)\|\leq 1$ for all $\lambda \in \mathbb{D}$. Also, $\mathcal{G}_{V}(\lambda)$ is analytic on $\mathbb{D}$ and hence $\mathcal{G}_{V}\in \mathcal S_{1}(\mathbb C^3,\mathbb C^3)$. Now if $\Xi=\mathcal{G}_{L}$, then it satisfy all the required property. 
	\end{proof}

	Note that the function $\Xi$ constructed in the above procedure is not unique as the functions $f_1,f_2,g$ are not uniquely determined. The following proposition tells us the same.
	\begin{prop}
		Assume that $(N^{(1)},N^{(2)},N^{(3)})\in \tilde{R}_{11}$. Suppose that $f_1, \hat{f_1} \in \mathcal{H}_{N^{(1)}}$, $f_2, \hat{f_2} \in \mathcal{H}_{N^{(2)}}$, $ v_{\lambda,\textbf{z}}^1,\\ v_{\lambda,\textbf{z}}^2 \in \mathcal{H}_{N^{(3)}}$ and $g, \hat{g} \in H_{K_{(N^{(1)},N^{(2)},N^{(3)})}}$ with
		
		\begin{equation}\label{eq1}
			N^{(i)}(\lambda,\textbf{z},\mu,\textbf{w})=\overline{f_i(\mu,\textbf{w})} f_i(\lambda,\textbf{z})=\overline{\hat{f_i}(\mu,\textbf{w})} \hat{f_i}(\lambda,\textbf{z}),\,\, i=1,2
		\end{equation}
		\begin{equation}\label{eq2}
			N^{(3)}(\lambda,\textbf{z},\mu,\textbf{w})=\left\langle v_{\lambda,\textbf{z}}^1, v_{\mu,\textbf{w}}^1\right\rangle_{\mathcal{H}_{N^{(3)}}}=\left\langle v_{\lambda,\textbf{z}}^2, v_{\mu,\textbf{w}}^2\right\rangle_{\mathcal{H}_{N^{(3)}}}
		\end{equation}
		and
		\begin{equation}\label{eq3}
			K_{(N^{(1)},N^{(2)},N^{(3)})}(\lambda,\textbf{z},\mu,\textbf{w})=\overline{g(\mu,\textbf{w})} g(\lambda,\textbf{z})=\overline{\hat{g}(\mu,\textbf{w})} \hat{g}(\lambda,\textbf{z})	
		\end{equation}
		for all $z_1,z_2,\lambda,w_1,w_2, \mu \in \mathbb{D}$. With the functions $f_1,f_2,g,v^1$ and $\hat{f_1},\hat{f_2},\hat{g},v^2$, respectively, let $\Xi_1(\lambda)$ and $\Xi_2(\lambda)$ be contructed from $(N^{(1)},N^{(2)},N^{(3)})$ using $U W$ procedure. Then we have
		$$
		\Xi_2(\lambda)=\left[\begin{array}{ccc}
			\eta_1 & 0 & 0 \\
			0 & \eta_2 & 0 \\
			0 & 0 & \eta_3
		\end{array}\right] \Xi_1(\lambda)\left[\begin{array}{ccc}
			1 & 0 & 0 \\
			0 & \bar{\eta_2} & 0 \\
			0 & 0 & \bar{\eta_3}
		\end{array}\right]
		$$
		for some $\eta_1,\eta_2,\eta_3 \in \mathbb{T}$.
	\end{prop}
	\begin{proof}
		It is evident from the equation \eqref{eq1} and \eqref{eq3} that  $\hat{f_1}=\eta_2 f_1$, $\hat{f_2}=\eta_3 f_2$ and $\hat{g_1}=\eta_1 g_1$ respectively for some $\eta_1,\eta_2,\eta_3 \in \mathbb{T}$. By using Theorem \eqref{UW}, there exist $\Xi_1(\lambda),\Xi_2(\lambda) \in \mathcal S_{1}(\mathbb C^3,\mathbb C^3)$ which satisfy the following conditions:
		\begin{equation}\label{eq4}
			\Xi_1(\lambda)\left(\begin{array}{ccc}
				1\\ z_{1}f_{1}(\lambda,\textbf{z}) \\ z_{2}f_{2}(\lambda,\textbf{z})
			\end{array}\right)=\left(\begin{array}{ccc}
				g(\lambda,\textbf{z})\\ f_{1}(\lambda,\textbf{z}) \\ f_{2}(\lambda,\textbf{z})
			\end{array}\right)~~\text{and}~~\Xi_2(\lambda)\left(\begin{array}{ccc}
				1\\ z_{1}\hat{f_1}(\lambda,\textbf{z}) \\ z_{2}\hat{f_2}(\lambda,\textbf{z})
			\end{array}\right)=\left(\begin{array}{ccc}
				\hat{g}(\lambda,\textbf{z})\\ \hat{f_1}(\lambda,\textbf{z}) \\ \hat{f_2}(\lambda,\textbf{z})
			\end{array}\right)
		\end{equation}
		
		for all $\lambda \in \mathbb{D}$ and $\textbf{z}\in \mathbb{D}^{2}$. Hence from equation \eqref{eq4} we have
		\begin{equation}\label{eq5}
			\Xi_2(\lambda)\left(\begin{array}{ccc}
				1\\ z_{1}\hat{f_1}(\lambda,\textbf{z}) \\ z_{2}\hat{f_2}(\lambda,\textbf{z})
			\end{array}\right)=\Xi_2(\lambda)\left(\begin{array}{ccc}
				1 & 0 & 0 \\
				0 & \eta_2 & 0 \\
				0 & 0 & \eta_3
			\end{array}\right)\left(\begin{array}{ccc}
				1 \\ z_1f_{1}(\lambda,\textbf{z}) \\ z_2f_{2}(\lambda,\textbf{z})
			\end{array}\right)
		\end{equation}
		and  
		\begin{equation}\label{eq6}
			\left(\begin{array}{ccc}
				\hat{g}(\lambda,\textbf{z})\\ \hat{f_1}(\lambda,\textbf{z}) \\ \hat{f_2}(\lambda,\textbf{z})
			\end{array}\right)=\left(\begin{array}{ccc}
				\eta_1 & 0 & 0 \\
				0 & \eta_2 & 0 \\
				0 & 0 & \eta_3
			\end{array}\right)\left(\begin{array}{ccc}
				g(\lambda,\textbf{z})\\ f_{1}(\lambda,\textbf{z}) \\ f_{2}(\lambda,\textbf{z})
			\end{array}\right)=\left(\begin{array}{ccc}
				\eta_1 & 0 & 0 \\
				0 & \eta_2 & 0 \\
				0 & 0 & \eta_3
			\end{array}\right)\Xi_1(\lambda)\left(\begin{array}{ccc}
				1\\ z_{1}f_{1}(\lambda,\textbf{z}) \\ z_{2}f_{2}(\lambda,\textbf{z})
			\end{array}\right)
		\end{equation}
		for all $\lambda \in \mathbb{D}$ and $\textbf{z}\in \mathbb{D}^{2}$.  By using equation \eqref{eq4}, \eqref{eq5} and \eqref{eq6} it yields that 
		$$\left(\Xi_2(\lambda)\left(\begin{array}{ccc}
			1 & 0 & 0 \\
			0 & \eta_2 & 0 \\
			0 & 0 & \eta_3
		\end{array}\right)-\left(\begin{array}{ccc}
			\eta_1 & 0 & 0 \\
			0 & \eta_2 & 0 \\
			0 & 0 & \eta_3
		\end{array}\right)\Xi_1(\lambda)\right)\left(\begin{array}{ccc}
			1 \\ z_1f_{1}(\lambda,\textbf{z}) \\ z_2f_{2}(\lambda,\textbf{z})
		\end{array}\right)=0$$
		for all $\lambda \in \mathbb{D}$ and $\textbf{z}\in \mathbb{D}^{2}$. As $f_i~~i=1,2$ are nonzero analytic functions of three variables, the set of zeros of $f_i,~i=1,2$ is nowhere dense in $\mathbb{D}^{3}$. This shows that $$\Xi_2(\lambda)=\left(\begin{array}{ccc}
			\eta_1 & 0 & 0 \\
			0 & \eta_2 & 0 \\
			0 & 0 & \eta_3
		\end{array}\right) \Xi_1(\lambda)\left(\begin{array}{ccc}
			1 & 0 & 0 \\
			0 & \bar{\eta_2} & 0 \\
			0 & 0 & \bar{\eta_3}
		\end{array}\right)$$ for all $\lambda \in \mathbb{D}$. This completes the proof.
		
	\end{proof} 
	As a consequence of the above Proposition, we prove the following corollary.	
	\begin{cor}\label{cor}
		Let $\Xi$ be a function that is being constructed from $(N^{(1)},N^{(2)},N^{(3)})\in \tilde{R}_{11}$ by using the procedure $UW$. Then 
		$$\{\left[\begin{array}{ccc}
			\eta_1 & 0 & 0 \\
			0 & \eta_2 & 0 \\
			0 & 0 & \eta_3
		\end{array}\right] \Xi\left[\begin{array}{ccc}
			1 & 0 & 0 \\
			0 & \bar{\eta_2} & 0 \\
			0 & 0 & \bar{\eta_3}
		\end{array}\right]:\eta_1,\eta_2,\eta_3 \in \mathbb{T}\} \subseteq \mathcal S_{1}(\mathbb C^3,\mathbb C^3)$$
	\end{cor}
	
	\begin{defn}
		The map $Upper \,\,W:\tilde{R}_{11}\rightarrow \mathcal S_{1}(\mathbb C^3,\mathbb C^3)$ is given by 
		\begin{align*}Upper \,\,W(N^{(1)},N^{(2)},N^{(3)})&=\{\Xi\in \mathcal S_{1}(\mathbb C^3,\mathbb C^3)\,\, \text{constructed by procedure}\,\, UW \,\,\text{for}\,\, \\&(N^{(1)},N^{(2)},N^{(3)}) \in \tilde{R}_{11}\}.\end{align*}
	\end{defn}
	\begin{prop}\label{comp1}
		Let $(N^{(1)},N^{(2)},N^{(3)})\in \tilde{R}_{11}$ and  $\Xi\in Upper \,\,W(N^{(1)},N^{(2)},N^{(3)}).$ Then
		$$Upper \,\,E(\Xi)=(N^{(1)},N^{(2)},N^{(3)})$$
	\end{prop}
	\begin{proof}
		Suppose that $\Xi=((a_{ij}))_{i.j=1}^{3}\in \mathcal S_{1}(\mathbb C^3,\mathbb C^3).$ Then $Upper E(\Xi)=(N_{\Xi}^{(1)},N_{\Xi}^{(2)},N_{\Xi}^{(3)}),$ where $$N_{\Xi}^{(i)}(\lambda,\textbf{z},\mu,\textbf{w})=\overline{f_i(\mu,\textbf{w})} f_i(\lambda,\textbf{z}),i=1,2$$
		and 
		$$N_{\Xi}^{(3)}(\lambda,\textbf{z},\mu,\textbf{w})=\left[ 1~~\bar{w}_{1}\overline{f_{1}(\mu,\textbf{w})}~~\bar{w}_{2}\overline{f_{2}(\mu,\textbf{w})}\right]\frac{I-\Xi(\mu)^{*}\Xi(\lambda)}{1-\bar{\mu}\lambda}\left[\begin{array}{ccc}
			1\\ z_{1}f_{1}(\lambda,\textbf{z}) \\ z_{2}f_{2}(\lambda,\textbf{z})
		\end{array}\right]$$
		with
		$$f_{1}(\lambda,\textbf{z})=\frac{a_{21}(\lambda)-a_{21}(\lambda)a_{33}(\lambda)z_{2}+a_{23}(\lambda)a_{31}(\lambda)z_{2}}{1-a_{22}(\lambda)z_{1}-a_{33}(\lambda)z_{2}+(a_{22}(\lambda)a_{33}(\lambda)-a_{23}(\lambda)a_{32}(\lambda))z_{1}z_{2}}$$
		and 
		$$f_2(\lambda,\textbf{z})=\frac{a_{31}(\lambda)+a_{21}(\lambda)a_{32}(\lambda)z_{1}-a_{22}(\lambda)a_{31}(\lambda)z_{1}}{1-a_{22}(\lambda)z_{1}-a_{33}(\lambda)z_{2}+(a_{22}(\lambda)a_{33}(\lambda)-a_{23}(\lambda)a_{32}(\lambda))z_{1}z_{2}}$$
		for all $\lambda,\mu \in \mathbb{D}\,\, \text{and} \,\,\textbf{z},\textbf{w}\in \mathbb{D}^{2}$.
		By hypothesis, we note that $\Xi\in Upper~W(N^{(1)},N^{(2)},N^{(3)})$. Hence there exist functions $f_1,f_2$ and $g$ such that  \begin{equation}\label{eq18}
			N^{(i)}(\lambda,\textbf{z},\mu,\textbf{w})=\overline{f_i(\mu,\textbf{w})} f_i(\lambda,\textbf{z}),\,\, i=1,2,
		\end{equation}
		
		\begin{equation}\label{eq19}
			K_{(N^{(1)},N^{(2)},N^{(3)})}(\lambda,\textbf{z},\mu,\textbf{w})=\overline{g(\mu,\textbf{w})} g(\lambda,\textbf{z})	
		\end{equation}
		and 
		\begin{equation}\label{20}
			\Xi(\lambda)\left(\begin{array}{ccc}
				1\\ z_{1}f_{1}(\lambda,\textbf{z}) \\ z_{2}f_{2}(\lambda,\textbf{z})
			\end{array}\right)=\left(\begin{array}{ccc}
				g(\lambda,\textbf{z})\\ f_{1}(\lambda,\textbf{z}) \\ f_{2}(\lambda,\textbf{z})
			\end{array}\right)
		\end{equation} for all $z_1,z_2,\lambda,w_1,w_2, \mu \in \mathbb{D}$.
		It follows from \eqref{20} that
		\begin{equation}\label{21}
			a_{11}(\lambda)+z_{1}f_{1}(\lambda,\textbf{z})a_{12}(\lambda)+z_{2}f_{2}(\lambda,\textbf{z}) a_{13}(\lambda)=g(\lambda,\textbf{z}),  
		\end{equation}
		\begin{equation}\label{22}
			a_{21}(\lambda)+a_{22}(\lambda)z_{1}f_{1}(\lambda,\textbf{z})+a_{23}(\lambda)z_{2}f_{2}(\lambda,\textbf{z})=f_{1}(\lambda,\textbf{z})
		\end{equation}
		and
		\begin{equation}\label{23}
			a_{31}(\lambda)+a_{32}(\lambda)z_{1}f_{1}(\lambda,\textbf{z})+a_{33}(\lambda)z_{2}f_{2}(\lambda,\textbf{z})=f_{2}(\lambda,\textbf{z})	
		\end{equation}
		It yields from \eqref{22} and \eqref{23} that  $$1-a_{22}(\lambda)z_{1}-a_{33}(\lambda)z_{2}+(a_{22}(\lambda)a_{33}(\lambda)-a_{23}(\lambda)a_{32}(\lambda))z_{1}z_{2}\neq 0,$$ $$f_1(\lambda,\textbf{z})=\frac{a_{21}(\lambda)-a_{21}(\lambda)a_{33}(\lambda)z_{2}+a_{23}(\lambda)a_{31}(\lambda)z_{2}}{1-a_{22}(\lambda)z_{1}-a_{33}(\lambda)z_{2}+(a_{22}(\lambda)a_{33}(\lambda)-a_{23}(\lambda)a_{32}(\lambda))z_{1}z_{2}}$$
		and 
		$$f_2(\lambda,\textbf{z})=\frac{a_{31}(\lambda)+a_{21}(\lambda)a_{32}(\lambda)z_{1}-a_{22}(\lambda)a_{31}(\lambda)z_{1}}{1-a_{22}(\lambda)z_{1}-a_{33}(\lambda)z_{2}+(a_{22}(\lambda)a_{33}(\lambda)-a_{23}(\lambda)a_{32}(\lambda))z_{1}z_{2}}$$
		for all $z_1,z_2,\lambda \in \mathbb{D}$ and hence we have 
		\begin{equation} \label{NEX} N_{\Xi}^{(i)}(\lambda,\textbf{z},\mu,\textbf{w})=\overline{f_i(\mu,\textbf{w})} f_i(\lambda,\textbf{z})=N^{(i)}(\lambda,\textbf{z},\mu,\textbf{w}),\,\, i=1,2
		\end{equation}
		for all $\lambda,\mu \in \mathbb{D}\,\, \text{and} \,\,\textbf{z},\textbf{w}\in \mathbb{D}^{2}$. Furthermore, from \cite[Equation 2.10]{pal1}, we deduce that
		$$\mathcal{G}_{\Xi(\lambda)}(\textbf{z})=a_{11}(\lambda)+z_{1}f_{1}(\lambda,\textbf{z})a_{12}(\lambda)+z_{2}f_{2}(\lambda,\textbf{z}) a_{13}(\lambda)=g(\lambda,\textbf{z})$$
		for all $\lambda,\mu \in \mathbb{D}\,\, \text{and} \,\,\textbf{z},\textbf{w}\in \mathbb{D}^{2}$. Thus, we have 
		\begin{equation} \label{KNN} K_{(N^{(1)},N^{(2)},N^{(3)})}(\lambda,\textbf{z},\mu,\textbf{w})=\overline{g(\mu,\textbf{w})} g(\lambda,\textbf{z})=\overline{\mathcal{G}_{\Xi(\mu)}(\textbf{w})}\mathcal{G}_{\Xi(\lambda)}(\textbf{z})\end{equation} for all $\lambda,\mu \in \mathbb{D}\,\, \text{and} \,\,\textbf{z},\textbf{w}\in \mathbb{D}^{2}$. It implies from \cite[Proposition 2.11]{pal1} that 
		$$1-\overline{\mathcal{G}_{\Xi(\mu)}(\textbf{w})}\mathcal{G}_{\Xi(\lambda)}(\textbf{z})=(1-\bar{w}_1z_1)N_{\Xi}^{(1)}(\lambda,\textbf{z},\mu,\textbf{w})+(1-\bar{w}_2z_2)N_{\Xi}^{(2)}(\lambda,\textbf{z},\mu,\textbf{w})+(1-\bar{\mu}\lambda)N_{\Xi}^{(3)}(\lambda,\textbf{z},\mu,\textbf{w})$$
		and hence 
		\begin{align}\label{KNN1}1-K_{(N^{(1)},N^{(2)},N^{(3)})}(\lambda,\textbf{z},\mu,\textbf{w})\nonumber&=(1-\bar{w}_1z_1)N_{\Xi}^{(1)}(\lambda,\textbf{z},\mu,\textbf{w})+(1-\bar{w}_2z_2)N_{\Xi}^{(2)}(\lambda,\textbf{z},\mu,\textbf{w})\\ &+(1-\bar{\mu}\lambda)N_{\Xi}^{(3)}(\lambda,\textbf{z},\mu,\textbf{w})\end{align} for all $\lambda,\mu \in \mathbb{D}\,\, \text{and} \,\,\textbf{z},\textbf{w}\in \mathbb{D}^{2}$. By hypothesis, we observe that
		\begin{align}\label{KNN2}K_{(N^{(1)},N^{(2)},N^{(3)})}(\lambda,\textbf{z},\mu,\textbf{w})\nonumber&=1-(1-\bar{w}_1z_1)N^{(1)}(\lambda,\textbf{z},\mu,\textbf{w})-(1-\bar{w}_2z_2)N^{(2)}(\lambda,\textbf{z},\mu,\textbf{w})\\ &-(1-\bar{\mu}\lambda)N^{(3)}(\lambda,\textbf{z},\mu,\textbf{w})
		\end{align} for all $\lambda,\mu \in \mathbb{D}\,\, \text{and} \,\,\textbf{z},\textbf{w}\in \mathbb{D}^{2}$. Thus from \eqref{NEX},\eqref{KNN}, \eqref{KNN1} and \eqref{KNN2}, we conclude that  $N_{\Xi}^{(3)}(\lambda,\textbf{z},\mu,\textbf{w})=N^{(3)}(\lambda,\textbf{z},\mu,\textbf{w})$ for all $\lambda,\mu \in \mathbb{D}\,\, \text{and} \,\,\textbf{z},\textbf{w}\in \mathbb{D}^{2}$. This completes the proof.
	\end{proof}
	
	\begin{prop}
		Let $F\in \mathcal S_{1}(\mathbb C^3,\mathbb C^3)$ be such that $\left(\begin{array}{cc}
			F_{21}(\lambda)\\
			F_{31}(\lambda)
		\end{array}\right)\neq \bold{0}$. Then
		$$Upper\, W \circ Upper\,E(F)=\{\left[\begin{array}{ccc}
			\eta_1 & 0 & 0 \\
			0 & \eta_2 & 0 \\
			0 & 0 & \eta_3
		\end{array}\right] F\left[\begin{array}{ccc}
			1 & 0 & 0 \\
			0 & \bar{\eta_2} & 0 \\
			0 & 0 & \bar{\eta_3}
		\end{array}\right]:\eta_1,\eta_2,\eta_3 \in \mathbb{T}\}$$	
	\end{prop}
	\begin{proof}
		Suppose that $F=((F_{ij}))_{i,j=1}^{3}\in \mathcal S_{1}(\mathbb C^3,\mathbb C^3).$ Then by the definition of upper map we get $UpperE(F)=(N^{(1)}_F,N^{(2)}_F,N^{(3)}_F),$ where $$N_{F}^{(i)}(\lambda,\textbf{z},\mu,\textbf{w})=\overline{f_i(\mu,\textbf{w})} f_i(\lambda,\textbf{z}),i=1,2$$
		and 
		$$N_{F}^{(3)}(\lambda,\textbf{z},\mu,\textbf{w})=\left[ 1~~\bar{w}_{1}\overline{f_{1}(\mu,\textbf{w})}~~\bar{w}_{2}\overline{f_{2}(\mu,\textbf{w})}\right]\frac{I-F(\mu)^{*}F(\lambda)}{1-\bar{\mu}\lambda}\left[\begin{array}{ccc}
			1\\ z_{1}f_{1}(\lambda,\textbf{z}) \\ z_{2}f_{2}(\lambda,\textbf{z})
		\end{array}\right]$$
		with
		$$f_{1}(\lambda,\textbf{z})=\frac{F_{21}(\lambda)-F_{21}(\lambda)F_{33}(\lambda)z_{2}+F_{23}(\lambda)F_{31}(\lambda)z_{2}}{1-F_{22}(\lambda)z_{1}-F_{33}(\lambda)z_{2}+(F_{22}(\lambda)F_{33}(\lambda)-F_{23}(\lambda)F_{32}(\lambda))z_{1}z_{2}}$$
		and 
		$$f_2(\lambda,\textbf{z})=\frac{F_{31}(\lambda)+F_{21}(\lambda)F_{32}(\lambda)z_{1}-F_{22}(\lambda)F_{31}(\lambda)z_{1}}{1-F_{22}(\lambda)z_{1}-F_{33}(\lambda)z_{2}+(F_{22}(\lambda)F_{33}(\lambda)-F_{23}(\lambda)F_{32}(\lambda))z_{1}z_{2}}$$
		for all $\lambda,\mu \in \mathbb{D}\,\, \text{and} \,\,\textbf{z},\textbf{w}\in \mathbb{D}^{2}$. By \cite[Proposition 2.11]{pal1}, it yields that
		$$1-\overline{\mathcal{G}_{F(\mu)}(\textbf{w})}\mathcal{G}_{F(\lambda)}(\textbf{z})=(1-\bar{w}_1z_1)N_{F}^{(1)}(\lambda,\textbf{z},\mu,\textbf{w})+(1-\bar{w}_2z_2)N_{F}^{(2)}(\lambda,\textbf{z},\mu,\textbf{w})+(1-\bar{\mu}\lambda)N_{F}^{(3)}(\lambda,\textbf{z},\mu,\textbf{w})$$ and hence we have
		\begin{equation}
			\begin{aligned}
				K_{(N_F^{(1)},N_F^{(2)},N_F^{(3)})}(\textbf{z},\lambda,\textbf{w},\mu)&=1-(1-\bar{w}_1z_1)N_{F}^{(1)}(\lambda,\textbf{z},\mu,\textbf{w})+(1-\bar{w}_2z_2)N_{F}^{(2)}(\lambda,\textbf{z},\mu,\textbf{w})\\&+(1-\bar{\mu}\lambda)N_{F}^{(3)}(\lambda,\textbf{z},\mu,\textbf{w})\\
				&= \overline{\mathcal{G}_{F(\mu)}(\textbf{w})}\mathcal{G}_{F(\lambda)}(\textbf{z})
			\end{aligned}
		\end{equation}
		for all $\lambda,\mu \in \mathbb{D}\,\, \text{and} \,\,\textbf{z},\textbf{w}\in \mathbb{D}^{2}$. By applying Procedure $UW$ we can construct a function $\Xi\in \mathcal S_{1}(\mathbb C^3,\mathbb C^3)$ such that
		\begin{equation}\label{24}
			\Xi(\lambda)\left(\begin{array}{ccc}
				1\\ z_{1}f_{1}(\lambda,\textbf{z}) \\ z_{2}f_{2}(\lambda,\textbf{z})
			\end{array}\right)=\left(\begin{array}{ccc}
				\mathcal{G}_{F(\lambda)}(\textbf{z})\\ f_{1}(\lambda,\textbf{z}) \\ f_{2}(\lambda,\textbf{z})
			\end{array}\right).
		\end{equation}
		for all $\lambda,\mu \in \mathbb{D}\,\, \text{and} \,\,\textbf{z},\textbf{w}\in \mathbb{D}^{2}$.		
		We also notice that 
		\begin{equation}\label{25}
			\begin{aligned}
				F(\lambda)\left(\begin{array}{ccc}
					1\\ z_{1}f_{1}(\lambda,\textbf{z}) \\ z_{2}f_{2}(\lambda,\textbf{z})
				\end{array}\right)&=\left[\begin{array}{ccc}
					F_{11}(\lambda) & F_{12}(\lambda) & F_{13}(\lambda) \\
					F_{21}(\lambda) & F_{22}(\lambda) & F_{23}(\lambda) \\
					F_{31}(\lambda) & F_{32}(\lambda) & F_{33}(\lambda)
				\end{array}\right]\left(\begin{array}{ccc}
					1\\ z_{1}f_{1}(\lambda,\textbf{z}) \\ z_{2}f_{2}(\lambda,\textbf{z})
				\end{array}\right)\\
				&=\left(\begin{array}{ccc}
					\mathcal{G}_{F(\lambda)}(\textbf{z})\\ f_{1}(\lambda,\textbf{z}) \\ f_{2}(\lambda,\textbf{z})
				\end{array}\right)
			\end{aligned}
		\end{equation} 
		for all $\lambda,\mu \in \mathbb{D}\,\, \text{and} \,\,\textbf{z},\textbf{w}\in \mathbb{D}^{2}$. From equation \eqref{24} and \eqref{25} we see that $$(\Xi(\lambda)-F(\lambda))\left(\begin{array}{ccc}
			1\\ z_{1}f_{1}(\lambda,\textbf{z}) \\ z_{2}f_{2}(\lambda,\textbf{z})
		\end{array}\right)=0$$
		for all $\lambda,\mu \in \mathbb{D}\,\, \text{and} \,\,\textbf{z},\textbf{w}\in \mathbb{D}^{2}$.
		As  $\left(\begin{array}{cc}
			F_{21}\\
			F_{31}
		\end{array}\right)$ is  a nonzero analytic function on $\mathbb D,$ the zeros of 	$F_{21}$ and $F_{31}$ are isolated on $\mathbb D.$ Thus $\Xi(\lambda)=F(\lambda)$ for every $\lambda\in \mathbb{D}$. We obtain from corollary \ref{cor} that  $$Upper~ W(N^{(1)}_F,N^{(2)}_F,N^{(3)}_F)=\{\left[\begin{array}{ccc}
			\eta_1 & 0 & 0 \\
			0 & \eta_2 & 0 \\
			0 & 0 & \eta_3
		\end{array}\right] \Xi\left[\begin{array}{ccc}
			1 & 0 & 0 \\
			0 & \bar{\eta_2} & 0 \\
			0 & 0 & \bar{\eta_3}
		\end{array}\right]:\eta_1,\eta_2,\eta_3 \in \mathbb{T}\}.$$ Therefore, we have 
		\begin{align*}
			Upper\, W \circ Upper\,E(F)&= Upper~W(N^{(1)}_F,N^{(2)}_F,N^{(3)}_F)\\ &=\{\left[\begin{array}{ccc}
				\eta_1 & 0 & 0 \\
				0 & \eta_2 & 0 \\
				0 & 0 & \eta_3
			\end{array}\right] F\left[\begin{array}{ccc}
				1 & 0 & 0 \\
				0 & \bar{\eta_2} & 0 \\
				0 & 0 & \bar{\eta_3}
			\end{array}\right]:\eta_1,\eta_2,\eta_3 \in \mathbb{T}\}.
		\end{align*}
		This completes the proof.
	\end{proof}
	\begin{defn}
		The map $Right ~S$ is the set valued map from $\tilde{\mathcal{R}_1}$ to $\mathcal S_{3}(\mathbb C,\mathbb C)$, that is, $Right~S:\tilde{\mathcal{R}_1}\rightarrow \mathcal S_{3}(\mathbb C,\mathbb C)$ is given by 
		$$Right\,\,S(N^{(1)},N^{(2)},N^{(3)})=\{f\in \mathcal S_{3}(\mathbb C,\mathbb C):K_{(N^{(1)},N^{(2)},N^{(3)})}(\lambda,\textbf{z},\mu,\textbf{w})=\overline{f(\textbf{w},\mu)}f(\textbf{z},\lambda),\,\, \textbf{z},\textbf{w}\in \mathbb{D}^{2}\mu,\lambda\in \mathbb{D}\}$$ for every $(N^{(1)},N^{(2)},N^{(3)})\in \tilde{\mathcal{R}_1}$.
	\end{defn}
	
	\begin{prop}
		Let $(N^{(1)},N^{(2)},N^{(3)})\in \tilde{\mathcal{R}_1}.$ Then the map $\text{Right ~S}$ is well defined and
		$$\text{Right ~S}(N^{(1)},N^{(2)},N^{(3)})=\{\eta f:\eta\in \mathbb{T}\},$$
		where $f:\mathbb{D}^3\rightarrow\mathbb{C}$ is analytic and satisfies 
		$$K_{(N^{(1)},N^{(2)},N^{(3)})}(\lambda,\textbf{z},\mu,\textbf{w})=\overline{f(\textbf{w},\mu)}f(\textbf{z},\lambda)$$ for all $\mu,\lambda~\in \mathbb{D}$ and $\textbf{z},\textbf{w}\in \mathbb{D}^{2}.$
	\end{prop}
	\begin{proof}
		Suppose that $(N^{(1)},N^{(2)},N^{(3)})\in \tilde{\mathcal{R}_1}$. Then  $K_{(N^{(1)},N^{(2)},N^{(3)})}(\lambda,\textbf{z},\mu,\textbf{w})$ is an analytic kernel on $\mathbb D^3$ of rank $1.$ As a result, there exist an analytic function $f:\mathbb{D}^{3}\rightarrow\mathbb{C}$ such that 
		$$K_{(N^{(1)},N^{(2)},N^{(3)})}(\lambda,\textbf{z},\mu,\textbf{w})=\overline{f(\textbf{w},\mu)}f(\textbf{z},\lambda)$$ for all $\mu,\lambda~\in \mathbb{D}$ and $\textbf{z},\textbf{w}\in \mathbb{D}^{2}.$ Furthermore, if $g:\mathbb{D}^{3}\rightarrow\mathbb{C}$ is the analytic function which satisfies the following relation   
		$$K_{(N^{(1)},N^{(2)},N^{(3)})}(\lambda,\textbf{z},\mu,\textbf{w})=\overline{g(\textbf{w},\mu)}g(\textbf{z},\lambda)$$ for all $\mu,\lambda~\in \mathbb{D}$ and $\textbf{z},\textbf{w}\in \mathbb{D}^{2},$ then we have $g=\eta f$ for some $\eta \in \mathbb{T}.$ Observe that 
		\begin{align}\label{KNN12}1-K_{(N^{(1)},N^{(2)},N^{(3)})}(\lambda,\textbf{z},\mu,\textbf{w})\nonumber&=(1-\bar{w_{1}}z_{1})N^{(1)}(\textbf{z},\lambda,\textbf{w},\mu)-(1-\bar{w_{2}}z_{2})N^{(2)}(\textbf{z},\lambda,\textbf{w},\mu)\\\nonumber &-(1-\bar{\mu}\lambda)N^{(3)}(\textbf{z},\lambda,\textbf{w},\mu)\\ &\geq0\end{align} for all $\mu,\lambda~\in \mathbb{D}$ and $\textbf{z},\textbf{w}\in \mathbb{D}^{2}.$ Therefore, from \eqref{KNN12} it follows that $$1-\overline{f(\textbf{w},\mu)}f(\textbf{z},\lambda)=1-K_{(N^{(1)},N^{(2)},N^{(3)})}(\lambda,\textbf{z},\mu,\textbf{w})\geq0$$ for all $\mu,\lambda~\in \mathbb{D}$ and $\textbf{z},\textbf{w}\in \mathbb{D}^{2}.$ So, $|f(\textbf{w},\mu)|\leq1$ for all $\textbf{z}\in \mathbb{D}^{2}$ and $\lambda\in \mathbb{D}$. This shows that $f\in \mathcal S_{3}(\mathbb C,\mathbb C)$ and hence the $\text{Right ~S}$ is well defined. This completes the proof.
	\end{proof}
	Now we  discuss the relation of the map Right S and other maps in the rich structure.
	\begin{prop}\label{comp}
		Assume that $F\in \mathcal S_{1}(\mathbb C^3,\mathbb C^3)$. Then 
		$$Right\,S \circ Upper\,E(F)=\{\eta SE(F):\eta \in \mathbb{T}\}$$
	\end{prop}
	\begin{proof}
		Assume that $(N^{(1)}_{F},N^{(2)}_{F},N^{(3)}_{F})\in \tilde{R}_{1}$. By using the Proposition \eqref{upper} and \eqref{upper1} along with the definition of $Upper~E$ map we have  $$K_{(N_F^{(1)},N_F^{(2)},N_F^{(3)})}(\textbf{z},\lambda,\textbf{w},\mu)=\overline{\mathcal G_{F(\mu)}(\textbf{w})}\mathcal G_{F(\lambda)}(\textbf{z})=(\overline{-\mathcal G_{F(\mu)}(\textbf{w})})(-\mathcal G_{F(\lambda)}(\textbf{z}))$$ for all $\mu,\lambda~\in \mathbb{D}$ and $\textbf{z},\textbf{w}\in \mathbb{D}^{2}.$ Also by the definition of $SE$ map we deduce that  $$Right\,S \circ Upper\,E(F)=Right~S(N_F^{(1)},N_F^{(2)},N_F^{(3)})=\{\eta \mathcal G_{F(\lambda)}(\textbf{z}):\eta \in \mathbb{T}\}=\{\eta SE(F):\eta \in \mathbb{T}\}.$$ This completes the proof.
	\end{proof}
	\begin{prop}
		Assume that $(N^{(1)},N^{(2)},N^{(3)})\in \tilde{R}_{11}$. Then 
		$$Right\,S(N^{(1)},N^{(2)},N^{(3)})=\{SE(F):\,\,F \in Upper \,W(N^{(1)},N^{(2)},N^{(3)})\}$$
	\end{prop}
	\begin{proof}
		Let $(N^{(1)},N^{(2)},N^{(3)})\in \tilde{R}_{11}$ and let $A\in \mathcal S_{1}(\mathbb C^3,\mathbb C^3)$ be the map constructed through the procedure of $UW$ for $(N^{(1)},N^{(2)},N^{(3)}).$ Then $$Upper~W(N^{(1)},N^{(2)},N^{(3)})=\{\left[\begin{array}{ccc}
			\eta_1 & 0 & 0 \\
			0 & \eta_2 & 0 \\
			0 & 0 & \eta_3
		\end{array}\right] A\left[\begin{array}{ccc}
			1 & 0 & 0 \\
			0 & \bar{\eta_2} & 0 \\
			0 & 0 & \bar{\eta_3}
		\end{array}\right]:\eta_1,\eta_2,\eta_3 \in \mathbb{T}\}$$
		Note that 		
		\begin{equation}
			\begin{aligned}
				SE\left(\left[\begin{array}{ccc}
					\eta_1 & 0 & 0 \\
					0 & \eta_2 & 0 \\
					0 & 0 & \eta_3
				\end{array}\right] A\left[\begin{array}{ccc}
					1 & 0 & 0 \\
					0 & \bar{\eta_2} & 0 \\
					0 & 0 & \bar{\eta_3}
				\end{array}\right]\right)(\textbf{z},\lambda)=&SE\left(\left[\begin{array}{ccc}
					\eta_1a_{11} & \eta_1\bar{\eta_2}a_{12} & \eta_1\bar{\eta_3}a_{13} \\
					\eta_2a_{21} & a_{22} & \eta_2\bar{\eta_3}a_{23} \\
					\eta_3a_{31} & \eta_3\bar{\eta_2}a_{32} & a_{33}
				\end{array}\right]\right)(\textbf{z},\lambda)\\
				=&\eta_1 SE(A)(\textbf{z},\lambda)	
			\end{aligned}
		\end{equation}
		for all $\lambda~\in \mathbb{D}$, $\textbf{z}\in \mathbb{D}^{2}$ and all $\eta_1,\eta_2,\eta_3\in \mathbb{T}$. Therefore, $$\{SE(F):~F\in Upper~W(N,M)\}=\{\eta SE(A):~\eta \in \mathbb{T}\}$$ It follows from Proposition \eqref{comp} and \eqref{comp1} that
		\begin{align}Right~S(N^{(1)},N^{(2)},N^{(3)})\nonumber&=Right\,S \circ Upper\,E(A)=\{\eta_1 SE(A):\eta_1 \in \mathbb{T}\}\\\nonumber&=\{SE(F):\,\,F \in Upper \,W(N^{(1)},N^{(2)},N^{(3)})\}.\end{align} This completes the proof.
	\end{proof}
	\newpage
	\section{Matricial Nevanlinna-Pick Problem}
	In this section, we show how the interpolation problems for $G_{E(3;3;1,1,1)}$ and $G_{E(3;2;1,2)}$ can be reduced to a standard matricial Nevanlinna-Pick problem.
			\begin{thm}\label{nevalina}
				Suppose ${\bf{x}}=(x_1,\ldots,x_7)\in \mathcal O(\mathbb D,\Gamma_{E(3;3;1,1,1)}).$ Then there exists a unique function  \begin{equation}\label{Fz12}\mathcal F^{(z_2)}{(\lambda)}=((F_{ij}^{(z_2)}(\lambda)))_{i,j=1}^{2}\in \mathcal S_{1}(\mathbb C^2,\mathbb C^2)\end{equation}  such that 
				\begin{equation}\label{Fz2} F^{(z_2)}_{11}(\lambda)=\frac{x_1(\lambda)-z_2x_3(\lambda)}{1-x_2(\lambda)z_2},F^{(z_2)}_{22}(\lambda)=\frac{x_4(\lambda)-z_2x_6(\lambda)}{1-x_2(\lambda)z_2}~{\rm{ and }}~\det(\mathcal{F}^{(z_2)}{(\lambda)})=\frac{x_5(\lambda)-z_2x_7(\lambda)}{1-x_2(\lambda)z_2}\end{equation} for all $z_2\in \mathbb D$
				and $\sup_{|z_2|<1}|F_{12}^{(z_2)}(\lambda)|=\sup_{|z_2|< 1}|F_{21}^{(z_2)}(\lambda)|$ almost everywhere on $\mathbb T$ and for fixed but arbitrary $z_2\in \mathbb D$ $F_{21}^{(z_2)}$ is either $0$ or outer, and $F_{21}^{(z_2)}(0)\geq 0.$ Moreover, for all $\mu,\lambda\in\mathbb D$ and for all $z_1,w_1\in\mathbb C$ and $z_2,w_2 \in \mathbb D$ such that

				\begin{align}
					1-\overline{\Psi^{(3)}({\bf{x}}(\mu),w_1,w_2)}\Psi^{(3)}({\bf{x}}(\lambda),z_1,z_2)\nonumber&=(1-\bar{w}_1z_1)\overline{\tilde{\gamma}^{(w_2)}(\mu,w_1)}\tilde{\gamma}^{(z_2)}(\lambda,z_1)\\&+{\tilde{\eta}^{(w_2)}(\mu,w_1)}^*(I-{F^{(w_2)}(\mu)}^*F^{(z_2)}(\lambda))\tilde{\eta}^{(z_2)}(\lambda,z_1),
				\end{align}
				where $\tilde{\gamma}^{(z_2)}(\lambda,z_1):=(1-F_{22}^{(z_2)}(\lambda)z_1)^{-1}F_{21}^{(z_2)}(\lambda),\tilde{\eta}^{(z_2)}(\lambda,z_1)=\left(\begin{smallmatrix} 1\\z_1\tilde{\gamma}^{(z_2)}(\lambda,z_1)\end{smallmatrix}\right)$ and  $$\Psi^{(3)}({\bf{x}}(\lambda),z_1,z_2)=\frac{x_1(\lambda)-x_3(\lambda)z_2-x_5(\lambda)z_1+x_7(\lambda)z_1z_2}{1-x_2(\lambda)z_2-x_4(\lambda)z_1+x_6(\lambda)z_1z_2}.$$
				
			\end{thm}
			\begin{proof}
				
				For ${\bf{x}}=(x_1,\ldots,x_7)\in \mathcal O(\mathbb D,\Gamma_{E(3;3;1,1,1)}),$ and $z_2\in \mathbb{D}$, set $y_1^{(z_2)}(\lambda):=\frac{x_1(\lambda)-z_2x_{3}(\lambda)}{1-x_2(\lambda)z_2}, y_2^{(z_2)}(\lambda):=\frac{x_4(\lambda)-z_2x_{6}(\lambda)}{1-x_2(\lambda)z_2}$ and $y_3^{(z_2)}(\lambda):=\frac{x_5(\lambda)-z_2x_{7}(\lambda)}{1-x_2(\lambda)z_2}$ for all $\lambda \in \mathbb{D}$.

				Note that if ${\bf{x}}=(x_1,\ldots,x_7)\in \mathcal O(\mathbb D,\Gamma_{E(3;3;1,1,1)}),$ then 
				$(y_1^{(z_2)},y_2^{(z_2)},y_3^{(z_2)}) \in \mathcal O(\mathbb D,\Gamma_{E(2;2;1,1)})$ for all $z_2 \in \mathbb D.$
				
				Choose $z_2\in \mathbb{D}$ and fix it, we first assume that $(y_1^{(z_2)},y_2^{(z_2)},y_3^{(z_2)})$ is triangular, that is, 
				$$y_1^{(z_2)}y_2^{(z_2)}=y_3^{(z_2)}.$$
				By the characterization of tetrablock [Theorem $2.4$,\cite{Abouhajar}], we have $|y_1^{(z_2)}(\lambda)|\leq 1 $ and $|y_2^{(z_2)}(\lambda)|\leq 1$ for all $\lambda \in \mathbb D.$ Then the function \begin{equation}
					\mathcal F^{(z_2)}{(\lambda)}=\begin{pmatrix} y_1^{(z_2)}(\lambda) & 0\\0 & y_2^{(z_2)}(\lambda) \end{pmatrix}\in \mathcal S_{1}(\mathbb C^2,\mathbb C^2),\;\;\mbox{for all}\;\; \lambda\in \mathbb{D}
				\end{equation}
				and satisfy the required properties \eqref{Fz12} and \eqref{Fz2}.
				
				We consider the case $y_1^{(z_2)}y_2^{(z_2)}\neq y_3^{(z_2)}$, the  $H^{\infty}$ function $y_1^{(z_2)}y_2^{(z_2)}-y_3^{(z_2)}$ is non-zero and has unique outer-inner factorisation, that is,
				$\phi^{(z_2)}e^{C_{z_2}}=y_1^{(z_2)}y_2^{(z_2)}-y_3^{(z_2)},$ where $e^{C_{z_2}}$ is outer with $e^{C_{z_2}}(0)\geq 0$ and $\phi^{(z_2)}$ is an inner function. Let 
				\begin{equation}\label{FZ123}
					\mathcal F^{(z_2)}:=\begin{pmatrix} y_1^{(z_2)} & \phi^{(z_2)}e^{\frac{1}{2}C_{z_2}}\\e^{\frac{1}{2}C_{z_2}} & y_2^{(z_2)} \end{pmatrix}.
				\end{equation}
				Clearly
				\begin{align*}
					\det\mathcal{F}^{(z_2)}&=y_1^{(z_2)}y_2^{(z_2)}-\phi^{(z_2)}e^{C_{z_2}}\\&=y_1^{(z_2)}y_2^{(z_2)}-y_1^{(z_2)}y_2^{(z_2)}+y_3^{(z_2)}\\&= y_3^{(z_2)}.
				\end{align*}
				and $|F^{(z_2)}_{12}|=e^{\Re \frac{1}{2}C_{z_2}}=|F^{(z_2)}_{21}|$ almost everywhere on $\mathbb T$, $F^{(z_2)}_{21}$ is outer and $F^{(z_2)}_{21}(0)\geq 0.$ Therefore, we deduce that the only matrix that satisfies the required properties \eqref{Fz12} and \eqref{Fz2} for fixed but arbitrary $z_2\in \mathbb D$.
				
				We now show that $\mathcal F^{(z_2)}{(\lambda)}=((F_{ij}^{(z_2)}(\lambda)))_{i,j=1}^{2}\in \mathcal S_{1}(\mathbb C^2,\mathbb C^2).$
				Notice that $\mathcal F^{(z_2)}$ is holomorphic on $\mathbb D.$ In order to complete this, we must demonstrate that $I-\mathcal F^{(z_2)}{(\lambda)}^*\mathcal F^{(z_2)}{(\lambda)}\geq 0$ for  $\lambda \in \mathbb D.$ To show this, it is sufficent to prove that for all $\lambda \in \mathbb D$, the diagonal entries of $I-\mathcal F^{(z_2)}{(\lambda)}^*\mathcal F^{(z_2)}{(\lambda)}$ are non-negative and $\det( I-\mathcal F^{(z_2)}{(\lambda)}^*\mathcal F^{(z_2)}{(\lambda)})\geq 0.$ As $|F^{(z_2)}_{12}|=e^{\Re \frac{1}{2}C_{z_2}}=|F^{(z_2)}_{21}|$ almost everywhere on $\mathbb T$ and $F^{(z_2)}_{12}F^{(z_2)}_{21}=y_1^{(z_2)}y_2^{(z_2)}-y_3^{(z_2)},$ we get 
				$$ |F^{(z_2)}_{12}(\lambda)|^2=|F^{(z_2)}_{21}(\lambda)|^2=|F^{(z_2)}_{12}(\lambda)F^{(z_2)}_{21}(\lambda)|=|y_1^{(z_2)}(\lambda)y_2^{(z_2)}(\lambda)-y_3^{(z_2)(\lambda)}|$$ almost everywhere on $\mathbb T$. For  almost every $\lambda \in \mathbb T,$ we have 
				\begin{align}
					\nonumber& I-\mathcal F^{(z_2)}{(\lambda)}^*\mathcal F^{(z_2)}{(\lambda)}\\&=\begin{pmatrix} 1-|y_1^{z_2}(\lambda)|^2-|y_1^{z_2}(\lambda)y_2^{z_2}(\lambda)-y_3^{z_2}(\lambda)| &-\overline{y_1^{z_2}(\lambda)}F^{(z_2)}_{12}(\lambda)-\overline{F^{(z_2)}_{12}(\lambda)}y_2^{z_2}(\lambda)\\-y_1^{z_2}(\lambda)\overline{F^{(z_2)}_{12}(\lambda)}-F^{(z_2)}_{12}(\lambda)\overline{y_2^{(z_2)}(\lambda)} & 1-|y_2^{(z_2)}(\lambda)|^2-|y_1^{(z_2)}(\lambda)y_2^{(z_2)}(\lambda)-y_3{(z_2)}(\lambda)|
					\end{pmatrix}
				\end{align}
				and 
				\begin{align}
					\det(I-\mathcal F^{(z_2)}{(\lambda)}^*\mathcal F^{(z_2)}{(\lambda)})\nonumber &=1-|y_1^{(z_2)}(\lambda)|^2-2|y_1^{(z_2)}(\lambda)y_2^{(z_2)}(\lambda)-y_3^{(z_2)}(\lambda)|\\&-|y_2^{(z_2)}(\lambda)|^2+|y_3^{(z_2)}(\lambda)|^2.
				\end{align}
				By using \cite[Corollary 2.40]{pal1}, we deduce that  for almost every $\lambda\in \mathbb T$ and hence by Maximum Modulus Principle, $\|\mathcal F^{(z_2)}{(\lambda)}\|\leq 1$  for all $\lambda \in \mathbb D.$ By \cite[Proposition 2.16]{pal1}, for any $\mathcal F^{(z_2)}{(\lambda)}=((F_{ij}^{(z_2)}(\lambda)))_{i,j=1}^{2}\in \mathcal S_{1}(\mathbb C^2,\mathbb C^2),$
				
				\begin{align}\label{AAAA norm12}
					1-\overline{\mathcal G_{\mathcal F^{(w_2)}{(\mu)}}}(w_1)\mathcal G_{\mathcal F^{(z_2)}{(\lambda)}}(z_1)\nonumber&=\overline{{\tilde{\gamma_1}^{(w_2)}(w_1)}}(1-\bar{w}_1z_1)\tilde{\gamma}_1^{(z_2)}(z_1)\\&+{\tilde{\eta}^{(w_2)}(w_1)}^*(I-\mathcal F^{(z_2)}{(\mu)}^*\mathcal F^{(z_2)}{(\lambda)})\tilde{\eta}^{(z_2)}(z_1)
				\end{align}
				for all $\lambda,\mu,z_2,w_2\in \mathbb D$ and $z_1,w_1\in \mathbb C$ such that $1-F_{22}^{(z_2)}(\lambda)z_1\neq 0$ and $1-F_{22}^{(w_2)}(\mu)w_1\neq 0.$ We notice that 
				\begin{align}
					\mathcal G_{\mathcal F^{(z_2)}{(\lambda)}}(z_1)\nonumber&=F_{11}^{(z_2)}(\lambda)+\frac{F^{(z_2)}_{12}(\lambda)F^{(z_2)}_{21}(\lambda)z_1}{1-F_{22}^{(z_2)}(\lambda)z_1}\\\nonumber &=y_1^{z_2}(\lambda)+\frac{y_1^{z_2}(\lambda) y_2^{z_2}(\lambda)-y_3^{z_2}(\lambda)z_1}{1-y_2^{z_2}(\lambda)z_1}\\\nonumber &=\frac{x_1(\lambda)-z_2x_3(\lambda)}{1-x_2(\lambda)z_2}+\frac{\frac{x_1(\lambda)-z_2x_3(\lambda)}{1-x_2(\lambda)z_2}\frac{x_4(\lambda)-z_2x_6(\lambda)}{1-x_2(\lambda)z_2}-\frac{x_5(\lambda)-z_2x_7(\lambda)}{1-x_2(\lambda)z_2}z_1}{1-\frac{x_4(\lambda)-z_2x_6(\lambda)}{1-x_2(\lambda)z_2}z_1}\\&=\Psi^{(3)}({\bf{x}}(\lambda),z_1,z_2)
				\end{align}
				for all $\lambda,z_2\in \mathbb D$ and $z_1\in \mathbb C$ such that $1-F_{22}^{(z_2)}(\lambda)z_1\neq 0.$ We define the functions 
				$\tilde{\gamma}_1^{(z_2)}$ and $\tilde{\eta}^{(z_2)}$ by equations \cite[Equation 2.30]{pal1} and \cite[Equation 2.31]{pal1}. Thus we have 
				\begin{align}
					1-\overline{\Psi^{(3)}({\bf{x}}(\mu),w_1,w_2)}\Psi^{(3)}({\bf{x}}(\lambda),z_1,z_2)\nonumber&=1-\overline{\mathcal G_{\mathcal F^{(w_2)}{(\mu)}}}(w_1)\mathcal G_{\mathcal F^{(z_2)}{(\lambda)}}(z_1)\\\nonumber&=(1-\bar{w}_1z_1)\overline{\tilde{\gamma}^{(w_2)}(\mu,w_1)}\tilde{\gamma}^{(z_2)}(\lambda,z_1)\\&+{\tilde{\eta}^{(w_2)}(\mu,w_1)}^*(I-{F^{(w_2)}(\mu)}^*F^{(z_2)}(\lambda))\tilde{\eta}^{(z_2)}(\lambda,z_1)
				\end{align}
				for all $\lambda,\mu,z_2,w_2\in \mathbb D$ and $z_1,w_1\in \mathbb C$ such that $1-F_{22}^{(z_2)}(\lambda)z_1\neq 0$ and $1-F_{22}^{(w_2)}(\mu)w_1\neq 0.$ This completes the proof.

			\end{proof}
			
			The property of mapping from $\mathcal O(\mathbb D,\Gamma_{E(3;3;1,1,1)})$ and membership in the Schur class are related via the above analytic matrix function $\mathcal F^{(z_2)}$, which was constructed in the above theorem. 
			\begin{thm}
				Let $\lambda_1,\ldots,\lambda_n$ be distinct points in $\mathbb D$ and let ${\bf{x}}_j=(x_{1j},\ldots,x_{7j})\in  \Gamma_{E(3;3;1,1,1)}$ for $j=1,\ldots,n.$ Then the following conditions are equivalent:
				\begin{enumerate}
					\item There exists an analytic function ${\bf{x}}:\mathbb D\to \Gamma_{E(3;3;1,1,1)}$ such that $${\bf{x}}(\lambda_j)=(x_{1j},\ldots,x_{7j}),1\leq j\leq n.$$
					
					\item There exists $b_{j}^{(z_2)},c_{j}^{(z_2)}$ satisfying 
					\begin{equation}\label{pickn}b_{j}^{(z_2)}c_{j}^{(z_2)}=x_{1j}^{(z_2)}x_{2j}^{(z_2)}-x_{3j}^{(z_2)}, 1\leq j\leq n\end{equation} such that the Nevanlinna-Pick interpolation problem with data
					\begin{equation}\label{picknv}\lambda_j\to \left(\begin{smallmatrix} x_{1j}^{(z_2)} & b_{j}^{(z_2)}\\c_{j}^{(z_2)} & x_{2j}^{(z_2)}\end{smallmatrix}\right)\end{equation} is solvable, where $$x_{1j}^{(z_2)}=\frac{x_{4j}-z_2x_{6j}}{1-z_2x_{2j}}, x_{2j}^{(z_2)}=\frac{x_{1j}-z_2x_{3j}}{1-z_2x_{2j}}~~{\rm{and}}~~x_{3j}^{(z_2)}=\frac{x_{5j}-z_2x_{7j}}{1-z_2x_{2j}}$$ for all $z_2\in \mathbb D.$
				\end{enumerate}
			\end{thm}	
			\begin{proof}
				Suppose $(1)$ holds. Then by Theorem \eqref{nevalina}, there exists a unique function  
				\begin{equation}\label{Fz1212}\mathcal F^{(z_2)}{(\lambda)}=((F_{ij}^{(z_2)}(\lambda)))_{i,j=1}^{2}\in \mathcal S_{1}(\mathbb C^2,\mathbb C^2)\end{equation}  such that 
				$$(F^{(z_2)}_{11}(\lambda),F^{(z_2)}_{22}(\lambda),\det(\mathcal{F}^{(z_2)}{(\lambda)})\in \Gamma_{E(2;2;1,1)}$$ for all $z_2\in \mathbb D$
				and $\sup_{|z_2|<1}|F_{12}^{(z_2)}(\lambda)|=\sup_{|z_2|< 1}|F_{21}^{(z_2)}(\lambda)|$ almost everywhere on $\mathbb T$ and for fixed but arbitrary $z_2\in \mathbb D,$ $F_{21}^{(z_2)}$ is either $0$ or outer, and $F_{21}^{(z_2)}(0)\geq 0.$ Let $b_{j}^{(z_2)}=F^{(z_2)}_{12}(\lambda_j), c_{j}^{(z_2)}=F^{(z_2)}_{21}(\lambda_j)$ for $1\leq j\leq n.$ Then 
				
				$$F^{(z_2)}(\lambda_j)=\left(\begin{smallmatrix} x_{1j}^{(z_2)} & b_{j}^{(z_2)}\\c_{j}^{(z_2)} & x_{2j}^{(z_2)}\end{smallmatrix}\right)$$
				and so $x_{3j}^{(z_2)}=x_{1j}^{(z_2)}x_{2j}^{(z_2)}-b_{j}^{(z_2)}c_{j}^{(z_2)}.$ Hence we have 
				$$b_{j}^{(z_2)}c_{j}^{(z_2)}=x_{1j}^{(z_2)}x_{2j}^{(z_2)}-x_{3j}^{(z_2)}.$$ This demonstrates that \eqref{pickn} are satisfied, and $F^{(z_2)}$ can solve the matricial Nevanlinna-Pick problem with the data \eqref{picknv} for this choice of $b_j ^{(z_2)}$ and $c_j ^{(z_2)}$.
				
				Suppose that  $b_{j}^{(z_2)},c_{j}^{(z_2)}$ exists such that \eqref{pickn} hold.
				Let the Nevanlinna-Pick interpolation problem with data in \eqref{picknv} be solvable with $2\times 2$ matrix  $\mathcal F^{(z_2)}{(\lambda)}=((F_{ij}^{(z_2)}(\lambda)))_{i,j=1}^{2}\in \mathcal S_{1}(\mathbb C^2,\mathbb C^2)$, that is, $\mathcal F^{(z_2)}{(\lambda)}$ is the $2\times 2$ matricial Schur function such that \begin{equation}\label{piccknv}\mathcal F^{(z_2)}{(\lambda_j)}= \left(\begin{smallmatrix} x_{1j}^{(z_2)} & b_{j}^{(z_2)}\\c_{j}^{(z_2)} & x_{2j}^{(z_2)}\end{smallmatrix}\right)~{\rm{for}}~1\leq j\leq n.\end{equation}
				For fixed but arbitrary $z_2\in \mathbb D,$ define an analytic function ${\bf{x}}:=\textbf{x}^{(z_2)}:\mathbb D\to \Gamma_{E(2;2;1,1)} \subseteq \Gamma_{E(3;3;1,1,1)}$ by 
				\begin{equation}
					\begin{aligned}
						x_{1}^{(z_2)}(\lambda)=F_{11}^{(z_2)}(\lambda)\\
						x_{2}^{(z_2)}(\lambda)=F_{22}^{(z_2)}(\lambda)\\
						x_{3}^{(z_2)}(\lambda)=F_{11}^{(z_2)}(\lambda)F_{22}^{(z_2)}(\lambda)-F_{12}^{(z_2)}(\lambda)F_{21}^{(z_2)}(\lambda)=\det(\mathcal F^{(z_2)}{(\lambda)})
					\end{aligned}
				\end{equation}
				Notice that the condition \eqref{pickn} are satisfied, for $1\leq j\leq n,$
				\begin{equation}
					\begin{aligned}
						x_{1j}^{(z_2)}=x_{1}^{(z_2)}(\lambda_j)=F_{11}^{(z_2)}(\lambda_j)\\
						x_{2j}^{(z_2)}=x_{2}^{(z_2)}(\lambda_j)=F_{22}^{(z_2)}(\lambda_j)\\
						x_{3j}^{(z_2)}=x_{3}^{(z_2)}(\lambda_j)=\det(\mathcal F^{(z_2)})(\lambda_j)=x_{1j}^{(z_2)}x_{2j}^{(z_2)}-b_{j}^{(z_2)}c_{j}^{(z_2)}
					\end{aligned}
				\end{equation}
				For fixed $\lambda_j\in \mathbb D$ we note that $(x_{1j}^{(z_2)},x_{2j}^{(z_2)},x_{3j}^{(z_2)})\in \Gamma_{E(2;2;1,1)}$ for all $z_2\in \mathbb D$ and by characterisation of $\Gamma_{E(3;3;1,1,1)}$ in \cite[Theorem 2.39]{pal1}, we have ${\bf{x}}(\lambda_j)=(x_{1j},\ldots,x_{7j})\in  \Gamma_{E(3;3;1,1,1)}$ for $1\leq j\leq n.$ This completes the proof.
			\end{proof}
			
			Using the similar argument as in Theorem \eqref{nevalina}, we also prove the following theorem. 
			\begin{thm}\label{nevalina1}
				Suppose ${\tilde{\bf{x}}}=(x_1,\ldots,x_5)\in \mathcal O(\mathbb D,\Gamma_{E(3;2;1,2)}).$ Then there exists a unique function  \begin{equation}\label{Fz7}\mathcal F^{(z)}{(\lambda)}=((F_{ij}^{(z)}(\lambda)))_{i,j=1}^{2}\in \mathcal S_{1}(\mathbb C^2,\mathbb C^2)\end{equation}  such that 
				\begin{equation}\label{Fz8} F^{(z)}_{11}(\lambda)=\frac{2x_1(\lambda)-zx_3(\lambda)}{2-x_2(\lambda)z},F^{(z)}_{22}(\lambda)=\frac{x_2(\lambda)-2zx_4(\lambda)}{2-x_2(\lambda)z}~{\rm{ and }}~\det(\mathcal{F}^{(z)}{(\lambda)})=\frac{x_3(\lambda)-2zx_5(\lambda)}{1-x_2(\lambda)z}\end{equation} for all $z\in \mathbb D$
				and $\sup_{|z|<1}|F_{12}^{(z)}(\lambda)|=\sup_{|z|< 1}|F_{21}^{(z)}(\lambda)|$ almost everywhere on $\mathbb T$ and for fixed but arbitrary $z\in \mathbb D$ $F_{21}^{(z)}$ is either $0$ or outer, and $F_{21}^{(z)}(0)\geq 0.$ Moreover, for all $\mu,\lambda\in\mathbb D$ and for all  $z,w \in \mathbb D$ such that
				\begin{align}
					1-\overline{\Psi_{(3)}({\bf{x}}(\mu),w)}\Psi_{(3)}({\bf{x}}(\lambda),z)\nonumber&=(1-\bar{w}z)\overline{\tilde{\gamma}^{(w)}(\mu,w_1)}\tilde{\gamma}^{(z)}(\lambda,z)\\&+{\tilde{\eta}^{(w)}(\mu,w)}^*(I-{F^{(w)}(\mu)}^*F^{(z)}(\lambda))\tilde{\eta}^{(z)}(\lambda,z),
				\end{align}
				where $\tilde{\gamma}^{(z)}(\lambda,z):=(1-F_{22}^{(z)}(\lambda)z)^{-1}F_{21}^{(z)}(\lambda),\tilde{\eta}^{(z)}(\lambda,z)=\left(\begin{smallmatrix} 1\\z\tilde{\gamma}^{(z)}(\lambda,z)\end{smallmatrix}\right)$ and  $$\Psi_{(3)}({\bf{x}}(\lambda),z)=\frac{x_1(\lambda)-x_3(\lambda)z+x_5(\lambda)z^2}{1-x_2(\lambda)z+x_4(\lambda)z^2}.$$
				
			\end{thm}
			\begin{thm}
				Let $\lambda_1,\ldots,\lambda_n$ be distinct points in $\mathbb D$ and let $\tilde{{\bf{x}}}_j=(x_{1j},\ldots,x_{5j})\in  \Gamma_{E(3;2;1,2)}$ for $j=1,\ldots,n.$ Then the following conditions are equivalent:
				\begin{enumerate}
					\item There exists an analytic function $\tilde{{\bf{x}}}:\mathbb D\to \Gamma_{E(3;2;1,2)}$ such that $$\tilde{{\bf{x}}}(\lambda_j)=(x_{1j},\ldots,x_{5j}),1\leq j\leq n.$$
					
					\item There exists $\tilde{b}_{j}^{(z)},\tilde{c}_{j}^{(z)}$ satisfying 
					\begin{equation}\label{pickn1}\tilde{b}_{j}^{(z)}\tilde{c}_{j}^{(z)}=p_{1j}^{(z)}p_{2j}^{(z)}-p_{3j}^{(z)}, 1\leq j\leq n\end{equation} such that the Nevanlinna-Pick interpolation problem with data
					\begin{equation}\label{picknv12}\lambda_j\to \left(\begin{smallmatrix} p_{1j}^{(z)} & \tilde{b}_{j}^{(z)}\\\tilde{c}_{j}^{(z)} & p_{2j}^{(z)}\end{smallmatrix}\right)\end{equation} is solvable, where $$p_{1j}^{(z)}=\frac{2x_{1j}-zx_{3j}}{2-zx_{2j}}, p_{2j}^{(z)}=\frac{x_{2j}-2zx_{3j}}{2-zx_{2j}}~~{\rm{and}}~~p_{3j}^{(z)}=\frac{x_{3j}-2zx_{5j}}{2-zx_{2j}}$$ for all $z_2\in \mathbb D.$
				\end{enumerate}
			\end{thm}

\textsl{Acknowledgements:}
The second-named author thanks CSIR for financial support and the third named author thankfully acknowledges the financial support provided by  the research project of SERB with
ANRF File Number: CRG/2022/003058, by the Science and Engineering Research Board (SERB),
Department of Science and Technology (DST), Government of India.

\addcontentsline{toc}{chapter}{Bibliography}
\bibliographystyle{plain}
\end{document}